\documentclass[preprint, 11pt]{elsarticle}

\usepackage[T1]{fontenc}
\usepackage[utf8]{inputenc}
\usepackage[hmargin=0.9in]{geometry}
\usepackage[babel,german=quotes]{csquotes}
\usepackage{subcaption}
\usepackage{graphicx}
\usepackage[colorlinks=true,citecolor=blue,urlcolor=blue]{hyperref}
\usepackage{caption}
\usepackage{amsmath}
\usepackage{amsthm}
\usepackage{amssymb}
\usepackage{bm, amsfonts}
\usepackage[dvipsnames]{xcolor}
\usepackage{soul,framed}
\usepackage{tikz}
\usepackage{bm}
\usepackage{makecell}
\usepackage{titlesec}
\usepackage{multirow}
\usepackage{indentfirst}
\usepackage{aligned-overset}
\newtheorem{theorem}{Theorem}[section]
\newtheorem{lemma}[theorem]{Lemma}

\theoremstyle{definition}

\theoremstyle{remark}
\newtheorem{remark}[theorem]{Remark}

\numberwithin{equation}{section}

\usepackage[notref,notcite,final]{showkeys} % Show labels.
\newcommand{\vertiii}[1]{{\left\vert\kern-0.25ex\left\vert\kern-0.25ex\left\vert #1 
		\right\vert\kern-0.25ex\right\vert\kern-0.25ex\right\vert}}

\colorlet{shadecolor}{green}

\titleformat{\paragraph}
{\normalfont\normalsize\itshape}{\theparagraph}{1em}{}
\titlespacing*{\paragraph}
{0pt}{3.25ex plus 1ex minus .2ex}{1.5ex plus .2ex}

\makeatletter
\def\ps@pprintTitle{%
	\let\@oddhead\@empty
	\let\@evenhead\@empty
	\def\@oddfoot{
		\footnotesize\itshape
		\ifx\@journal\@empty Elsevier
		\else\@journal\fi
		\hfill\today
	}%
	\let\@evenfoot\@oddfoot}
\makeatother

\begin{document}
	
	\begin{frontmatter}
	  \title{Strongly consistent low-dissipation WENO schemes for finite elements}

	        \author[TU]{Joshua Vedral\corref{cor1}}
		\ead{joshua.vedral@math.tu-dortmund.de}
		\author[LUT]{Andreas Rupp}
		\ead{andreas.rupp@fau.de}
		\cortext[cor1]{Corresponding author}
		\author[TU]{Dmitri Kuzmin}
		\ead{kuzmin@math.uni-dortmund.de}

	        \address[TU]{Institute of Applied Mathematics (LS III), TU Dortmund University\\ Vogelpothsweg 87, D-44227 Dortmund, Germany}
		\address[LUT]{School of Engineering Science, Lappeenranta--Lahti University of Technology\\ P.O. Box 20, 53851 Lappeenranta, Finland}
		
		\journal{Submitted to Applied Numerical Mathematics}
		
		\begin{abstract}
		  We propose a way to maintain strong consistency and facilitate error analysis in the context of dissipation-based WENO stabilization for continuous and discontinuous Galerkin discretizations of conservation laws. Following Kuzmin and Vedral (J. Comput. Phys. 487:112153, 2023) and Vedral (arXiv preprint arXiv:2309.12019), we use WENO shock detectors to determine appropriate amounts of low-order artificial viscosity. In contrast to existing WENO methods, our approach blends candidate polynomials using residual-based nonlinear weights. The shock-capturing terms of our stabilized Galerkin methods vanish if residuals do. This enables us to achieve improved accuracy compared to weakly consistent alternatives. As we show in the context of steady convection-diffusion-reaction (CDR) equations, nonlinear local projection stabilization terms can be included in a way that preserves the coercivity of local bilinear forms. For the corresponding Galerkin-WENO discretization of a CDR problem, we rigorously derive a priori error estimates. Additionally, we demonstrate the stability and accuracy of the proposed method through one- and two-dimensional numerical experiments for hyperbolic conservation laws and systems thereof. The numerical results for representative test problems are superior to those obtained with traditional WENO schemes, particularly in scenarios involving shocks and steep gradients.
		\end{abstract}
		
		
	\end{frontmatter}
	% ---------------------------------------------------------------------------
	\section{Introduction}
	% ---------------------------------------------------------------------------	
	It is well known that standard continuous Galerkin (CG) and discontinuous Galerkin (DG) discretizations of hyperbolic problems and convection-dominated convection-diffusion-reaction (CDR) problems suffer from spurious oscillations, instability, and convergence to wrong weak solutions in the nonlinear case. To address these challenges, numerous stabilized schemes have been proposed in the literature.
	
	Among the variety of stabilization techniques, local projection stabilization (LPS) methods \cite{braack2006, ganesan2008, knobloch2009, kuzmin2017, matthies2008, matthies2007} are a prominent  variant of variational multiscale (VMS) methods \cite{codina2018, hughes1998, john2006}. LPS methods address the limitations of residual-based stabilization techniques, like the streamline upwind Petrov-Galerkin (SUPG) methods \cite{brooks1982, burman2010}, i.e., the lack of symmetry and involving second-order derivatives in the residual due to the consistency property of the method. Initially introduced by Becker and Braack for the Stokes problem \cite{becker2001} and later extended to handle transport problems \cite{becker2004}, LPS methods are designed to stabilize standard Galerkin discretizations by controlling the fluctuations of gradients. This concept is akin to subgrid scale modeling \cite{ern2004,guermond1999,john2006,john2010} for large eddy simulation. Initial analysis of LPS was focused on low-order discretizations for the Oseen equations \cite{braack2006}, with further analysis conducted within the context of CDR equations; see, e.g., \cite{barrenechea2013, knobloch2009}.  

	High-order stabilization is generally insufficient to fully suppress spurious oscillations in the vicinity of steep gradients. To get rid of these oscillations in a way that preserves high-order accuracy in smooth regions, high-order baseline discretizations need to be equipped with nonlinear shock-capturing terms. The residual-based methods analyzed in \cite{knopp2002,lube2006a,lube2006} use streamline diffusion operators to achieve linear stability and nonlinear crosswind diffusion for shock-capturing purposes. The numerical analysis of such schemes is greatly facilitated by the independence and local coercivity of bilinear forms associated with the streamline and crosswind components of the stabilization terms.
       
  Alternatively, a blend of low-order monotone schemes and their high-order stabilized counterparts can be employed. This concept traces its origins back to the classical Jameson-Schmidt-Turkel (JST) scheme \cite{boris1973,jameson2017,jameson1981}. In this approach, first-order artificial viscosity is applied near discontinuities and replaced with high-order background dissipation in smoother regions. In the finite element context, the need for proper blending of low- and high-order stabilization terms was emphasized by Ern and Guermond \cite{ern2013}. Barrenechea et al. \cite{barrenechea2017c} analyzed a variable-order stabilized finite element method under the assumption that the nonlinear low-order diffusion operates only on a small fixed subdomain. Thus, they considered a linear scheme in which the switching parameter does not depend on the solution, but is solely a function of the space coordinate. Our work aims to extend this analysis, particularly considering the nonlinear coupling of high- and low-order stabilization terms. 

        Another prominent family of numerical methods achieves
        nonlinear stability by using special polynomial reconstructions. Many finite difference (FD), finite volume (FV) and DG schemes are based on the essentially non-oscillatory (ENO) paradigm developed by Harten et al. \cite{harten1987}. Rather than selecting the smoothest polynomial, Liu et al. \cite{liu1994} proposed to adaptively blend all polynomial approximations. The so-called weighted ENO (WENO) schemes ensure accurate representation of discontinuities, eliminate spurious oscillations and achieve high-order accuracy for smooth solutions. Qiu and Shu \cite{qiu2005} extended the WENO framework to Runge-Kutta discontinuous Galerkin (RKDG) methods by introducing a limiter based on WENO reconstructions. They further considered WENO schemes based on Hermite polynomials, termed Hermite WENO (HWENO) schemes \cite{qiu2004, qiu2005b}. By incorporating derivative values of neighboring elements into WENO reconstructions, the compactness of reconstruction stencils is optimized. For an overview of existing DG-WENO methods we refer the reader to \cite{shu2016, zhang2011}.

        The WENO approaches introduced in \cite{kuzmin2023a, vedral2023} belong to the family of dissipation-based shock-capturing techniques. However, they use a WENO-based smoothness indicator to construct a nonlinear blend of high-order and low-order stabilization terms. We analyze and improve such methods in the present paper. By using residual-based weights for candidate polynomials of a WENO reconstruction, we achieve excellent shock-capturing capabilities while retaining the optimal convergence rate of our baseline discretization method. Moreover, we perform rigorous theoretical studies of our nonlinear scheme. In particular, we prove stability, show existence of solutions to discrete problems, and obtain a priori estimates in the context of steady CDR equations. An important result of our analysis is an improved definition of the local stabilization operator. The proposed revision leads to a WENO scheme that consists of a linear LPS part and a locally coercive nonlinear component.
	
	We continue, in the next section, by presenting the generic form of stabilized CG and DG methods for hyperbolic conservation laws and systems thereof. In Section \ref{sec:stab2}, we consider a particular nonlinear blend of dissipative high-order and low-order stabilization terms. Following that, Section \ref{sec:res} introduces our new residual-based WENO scheme. Section \ref{sec:apriori} presents the results of our analysis for CG-WENO discretizations of steady CDR equations. Section \ref{sec:aprioriDG} extends them to the DG setting. In Section \ref{sec:num}, we conduct a series of numerical experiments for hyperbolic equations and systems thereof. In the last section, we draw conclusions and discuss some perspectives.
	
	% ---------------------------------------------------------------------------
	\section{Stabilized Galerkin methods}
	\label{sec:stab}
	% ---------------------------------------------------------------------------
	Let $u(\mathbf{x},t)\in \mathbb{R}^m$, $m\in \mathbb{N}$ be a vector of conserved quantities depending on the space location $\mathbf{x}$ and time instant $t\geq 0$. We consider the initial value problem
	\begin{subequations}
		\begin{alignat}{2}
		\frac{\partial u}{\partial t}+\nabla\cdot\mathbf{f}(u)&=0 \quad &&\text{in }\Omega \times \mathbb{R}_+, \label{eq:pde}\\
		u(\cdot,0)&=u_0 \quad &&\text{in } \Omega,
		\end{alignat}\label{eq:ivp}%
	\end{subequations}
	where $\Omega\in \mathbb{R}^d$, $d\in\{1,2,3\}$ is a bounded domain with Lipschitz boundary $\partial \Omega$, $\mathbf{f}(u)=(f_{ij})\in \mathbb{R}^{m\times d}$ is an array of inviscid fluxes with $(\nabla \cdot \mathbf{f})_i=\Big(\sum_{j=1}^{d}\frac{\partial f_{ij}}{\partial x_j}\Big)\in \mathbb{R}^{m}$ and $u_0\in\mathbb{R}^m$ is the initial datum. In the scalar ($m=1$) case, boundary conditions are imposed weakly at the inlet of the domain. For hyperbolic systems ($m>1$), appropriate choices of boundary conditions depend on the number of incoming and outgoing waves (see, e.g., \cite{guaily2013}).
	
	Let $\mathcal{T}_h=\{K_1,\ldots,K_{E_h}\}$ be a decomposition of the domain $\Omega$ into non-overlapping elements $K_e$, $e=1,\ldots,E_h$ such that $\cup_{K_e\in\mathcal{T}_h}=\bar{\Omega}$. We denote by $h=\max_{K_e\in\mathcal{T}_h}h_e$, where $h_e=\mathrm{diam}(K_e)$, the mesh size associated with $\mathcal{T}_h$. We discretize \eqref{eq:pde} in space using either the continuous (CG) or the discontinuous Galerkin (DG) finite element method. The corresponding finite element spaces read
	\begin{subequations}
	\begin{align*}
		V_h^c(\mathcal{T}_h)&=\{v_h\in C^0(\bar{\Omega}):v_h|_{K_e}\in V_h(K_e)\,\forall K_e\in\mathcal{T}_h\},\\
		V_h^d(\mathcal{T}_h)&=\{v_h\in L^2(\bar{\Omega}):v_h|_{K_e}\in V_h(K_e)\,\forall K_e\in\mathcal{T}_h\},
	\end{align*} 
	\end{subequations}
	respectively. Here, $V_h(K_e)\in \{\mathbb{P}_p(K_e), \mathbb{Q}_p(K_e)\}$ is the space of polynomials of (total) degree up to~$p$. For simplicity, we will write $V_h$ whenever possible, encompassing both $V_h^c$ and $V_h^d$. We seek an approximate solution
	\begin{align}
		u_h(\mathbf{x},t)=\sum_{j=1}^{N_h}u_j(t)\varphi_j(\mathbf x)
		\label{eq:sol}
	\end{align}
	in the finite element space $V_h$ spanned by basis functions $\varphi_1,\ldots,\varphi_{N_h}$. Our methodology presented below does not rely on a particular choice of basis functions. Popular choices include Lagrange, Bernstein and Legendre-Gauss-Lobatto (LGL) basis functions. 
	
	By inserting \eqref{eq:sol} into \eqref{eq:pde}, multiplying by a test function $w_h\in V_h$ and integrating over $\Omega$, we obtain a finite element discretization of \eqref{eq:ivp}. The semi-discrete problem corresponding to the standard continuous Galerkin method reads
	\begin{align}
	\sum_{e=1}^{E_h}\int_{K_e}w_h\Bigg(\frac{\partial u_h}{\partial t}+\nabla \cdot \mathbf{f}(u_h)\Bigg)\,\mathrm{d}\mathbf{x}=0\qquad \forall w_h\in V_h.
	\label{eq:galsd}
	\end{align}
  The DG version can be found in \cite{vedral2023}.
  %%%%%%
  % The details of the DG discretization are outlined in \cite{vedral2023}. We employ the local Lax-Friedrichs flux.
  %%%%%%
	For continuous finite elements, communication between elements is ensured via the continuous coupling of all elements. However, in the DG setting, the solution is typically discontinuous across element interfaces, necessitating the use of numerical flux functions. In our numerical experiments, we use the local Lax-Friedrichs numerical flux.
	
	To maintain generality, we introduce a local stabilization operator $s_h^e(\cdot,\cdot)$, which is yet to be defined. This operator allows for the representation of multiple schemes, such as LPS, SUPG, and others, through the following stabilized version of \eqref{eq:galsd}:
	\begin{align}
	\sum_{e=1}^{E_h}\int_{K_e}w_h\Bigg(\frac{\partial u_h}{\partial t}+\nabla \cdot \mathbf{f}(u_h)\Bigg)\,\mathrm{d}\mathbf{x}+\sum_{e=1}^{E_h}s_h^e(u_h,w_h)=0\qquad \forall w_h\in V_h.
	\end{align}
	The numerical solution can be evolved in time, e.g., using a strong stability preserving (SSP) Runge-Kutta method \cite{gottlieb2001}.
	\section{Dissipation-based nonlinear stabilization}
	\label{sec:stab2}
	In this section, we unify the stabilization techniques designed in \cite{kuzmin2023a} for continuous finite elements and in \cite{vedral2023} for their discontinuous counterparts in a broader Galerkin setting.

        In the DG context, numerical stability is guaranteed, e.g., for piecewise-constant (DG-$\mathbb{P}_0$) approximations with local Lax-Friedrichs fluxes. Obviously, the use of piecewise $\mathbb{P}_0$ discretizations is not an option for continuous Galerkin methods. To construct a low-order stabilization term that is suitable for CG and DG approaches alike,  we add isotropic artificial diffusion throughout the computational domain. The local stabilization operator that we use for this purpose is defined by \cite{kuzmin2023a,vedral2023} 
	\begin{align}
		s_h^{e,L}(u_h,w_h)=\nu_e\int_{K_e}\nabla w_h\cdot \nabla u_h\,\mathrm{d}\mathbf{x},
		\label{eq:lostab}
	\end{align}
	where
	\begin{align}\label{eq:nuLF}
		\nu_e=\frac{\lambda_eh_e}{2p}
	\end{align}
        is a Lax-Friedrichs-type artificial viscosity coefficient and $p$ denotes the polynomial degree. The involved parameter $\lambda_e=\|\mathbf{f}'(u_h)\|_{L^\infty(K_e)}$ is an upper bound for the maximum wave speed inside $K_e$. We refer the reader to \cite{guermond2016a} for estimates on the wave speeds for the Euler equations.
	
	Since stabilization of this kind is appropriate only in regions with steep gradients, we consider a symmetric version of the two-level VMS method presented by John et al. \cite{john2006}. To this end, we introduce a local fluctuation operator (as in the LPS methodology)
	\begin{subequations}
	\begin{alignat}{2}
		\kappa: \qquad &L^2(\Omega)^d\to L^2(\Omega)^d, \qquad &&v \mapsto v-P_hv, \label{eq:flucglobal}\\
		\kappa_e: \qquad &L^2(\Omega)^d\to L^2(K_e)^d, \qquad &&v \mapsto \kappa(v)|_{K_e}. 
	\end{alignat}\label{eq:fluctuation}%
	\end{subequations}
    Our choice of  $P_h:L^2(\Omega)\to V_h$ for continuous finite elements ($V_h=V_h^c$) is the (componentwise) Scott-Zhang variant \cite{scott1990} of the Cl\'ement operator \cite{clement1975}. In the DG ($V_h=V_h^d$) version, we define $P_h$ as the (componentwise) $L^2$ orthogonal projection. The design of the high-order stabilization term 
        	\begin{align}
        	\begin{split}
		s_h^{e,H}(u_h,w_h)&=\nu_e\int_{K_e}\kappa_e(\nabla w_h)\cdot \kappa_e(\nabla u_h)\,\mathrm{d}\mathbf{x}\\
		&=\nu_e\int_{K_e}(\nabla w_h - P_h\nabla w_h)\cdot (\nabla u_h -P_h\nabla u_h)\,\mathrm{d}\mathbf{x}
		\end{split}
		\label{eq:hostab}
	\end{align}
                builds on the concept of orthogonal subscale stabilization introduced by Codina and Blasco \cite{codina1997,codina2002}.

                The above choice of the projection operators is motivated by the analysis presented in Sec. \ref{sec:apriori}. Note that the high-order stabilization term \eqref{eq:hostab} vanishes in the DG case, in which the local $L^2$~projection yields $P_h\nabla u_h=\nabla u_h$ on $K_e$. In principle, we could project the gradient into the space $V_h^c$ using the Scott-Zhang operator also in the DG version, but there is no need for adding high-order stabilization in DG schemes equipped with stable numerical fluxes \cite{vedral2023}. In particular, the entropy stability \cite{jiang1994} and $L^2$ stability \cite{cockburn1998a} of such DG schemes can be shown in the scalar case. On the other hand, incorporating the bilinear form \eqref{eq:hostab} into the semi-discrete problem \eqref{eq:galsd} is essential for stability and high-order accuracy of CG approximations in regions where the solution is sufficiently smooth.

	Appropriate blending of the high-order stabilization term \eqref{eq:hostab} and the low-order stabilization term \eqref{eq:lostab} is needed to obtain stable and accurate numerical approximations. Introducing a blending factor $\gamma_e\in[0,1]$, our nonlinear stabilization can be written as (cf. \cite{barrenechea2017c,jameson2017,jameson1981,kuzmin2023a,vedral2023})
	\begin{align}
	\begin{split}
	s_h^e(u_h;u_h,w_h)&=\gamma_e(u_h)s_h^{e,H}(u_h,w_h)+(1-\gamma_e(u_h))s_h^{e,L}(u_h,w_h)\\
	&=\gamma_e(u_h)\nu_e\int_{K_e}\kappa_e(\nabla w_h)\cdot \kappa_e(\nabla u_h)\,\mathrm{d}\mathbf{x}+(1-\gamma_e(u_h))\nu_e\int_{K_e}\nabla w_h\cdot \nabla u_h\,\mathrm{d}\mathbf{x}.
	\end{split}
	\label{eq:blend}
	\end{align}
	Following the JST design principles, $\gamma_e$ should approach $0$ in cells with discontinuities and $1$ in cells where the solution is sufficiently smooth.
	
	There are several choices to consider for the blending factor. Traditional shock detectors rely on measures such as entropy \cite{guermond2011,krivodonova2004,lv2016}, total variation and slope \cite{ducros1999,harten1978,harten1983a,hendricks2018,jameson1981,ren2003} or (entropy) residual \cite{ern2013,marras2018,stiernstrom2021}. As an alternative, we consider smoothness sensors that utilize WENO-based reconstructions \cite{hill2004,li2020,movahed2013,visbal2005,wang2023,zhao2019,zhao2020}. Following our previous work \cite{kuzmin2023a,vedral2023}, we use
	\begin{align}\label{eq:gamma-old}
		\gamma_e (u_h)=1-\left(\min\Bigg\{1, \frac{\|u_h-u_h^{*}\|_e}{\|u_h\|_e}\Bigg\}\right)^q,
	\end{align}
	where $u_h^*$ is a WENO reconstruction and $q$ is a parameter that determines the sensitivity of the smoothness sensor $\gamma_e$ to the relative difference between $u_h$ and $u_h^*$. We measure this difference using the scaled Sobolev semi-norm \cite{friedrich1998,jiang1996}
	\begin{align}
		\|v\|_e=\Bigg(\sum_{1\leq |\mathbf{k}|\leq p}h_e^{2|\mathbf{k}|-d}\int_{K_e}|D^{\mathbf{k}}v|^2\,\mathrm{d}\mathbf{x}\Bigg)^{1/2} \qquad \forall v\in H^p(K_e),
		\label{eq:snorm}
	\end{align}
	where $\mathbf{k}=(k_1,\ldots,k_d)$ is the multiindex of the partial derivative
	\begin{align*}
		D^{\mathbf{k}}v=\frac{\partial^{|\mathbf{k}|}v}{\partial x_1^{k_1}\cdots \partial x_d^{k_d}}, \quad |\mathbf{k}|=k_1+\ldots+k_d.
	\end{align*}
	
	\begin{remark}
		Incorporating the scaling factors $h_e^{2|\mathbf{k}|-d}$ in \eqref{eq:snorm} is essential to prevent the leading coefficients of the highest-order derivatives from becoming dominant. The sensitivity of these coefficients can lead to strong oscillations in the solution. We refer to \cite{friedrich1998} for numerical experiments demonstrating the oscillations caused by omitting these scaling factors.
	\end{remark}
	\begin{remark}
		In DG schemes equipped with Hermite-type WENO limiters, the use of troubled cell indicators is a common strategy to minimize the computational cost associated with polynomial reconstructions; see, e.g., \cite{qiu2004,qiu2005b,qiu2005,zhu2016}. Among these indicators are minmod-type total variation bounded (TVB) limiters \cite{cockburn1989}, which flag cells as troubled whenever they induce a change in slope. Another prominent example is the KXRCF shock indicator developed in \cite{krivodonova2004}. However, these indicators often misidentify smooth cells, particularly those near smooth extrema, as troubled. To demonstrate that our methodology maintains high-order accuracy across all cells, we choose not to employ any troubled cell indicator.
	\end{remark}
	% ---------------------------------------------------------------------------
	\section{Residual-based WENO scheme}
	\label{sec:res}
	
        %	We proceed with the derivation of strong consistency preserving (SCP) smoothness sensors $\gamma_e$ for our dissipation-based WENO scheme.
        
        The standard WENO averaging for finite volume \cite{friedrich1998,jiang1996} and discontinuous Galerkin \cite{luo2007,zhong2013,zhu2009,zhu2017} methods produces a convex combination $u_h^{*}$ of candidate polynomials. As a rule, smaller weights are assigned to oscillatory polynomials with large derivatives. However, this criterion does not generally prevent unnecessary modifications of the baseline finite element scheme, e.g., in situations when exact solutions belong to the finite element space and the method is consistent.
		
%	To ensure that the inconsistent low-order stabilization is completely switched off in cells with small residuals, we replace the original definition \eqref{eq:gamma-old} of our smoothness sensor with 
%	\begin{align}\label{eq:gamma-new}
%		\gamma_e=1-\left(\min\Bigg\{1, \frac{\min\{C_RR^e,\|u_h-u_h^{*}\|_e\}}{\|u_h\|_e}\Bigg\}\right)^q,
%	\end{align}
%        where $R^e$ is the element residual (as defined in the next section for CDR equations) and $C_R>0$ is a scaling parameter. We remark that there exists a convex average $u_h^{**}$ of $u_h$ and $u_h^{*}$ such that
%        $$\|u_h-u_h^{**}\|_e=
%        \min\{C_RR^e,\|u_h-u_h^{*}\|_e\}.
%        $$
%        Thus, the effect of replacing \eqref{eq:gamma-old} by  \eqref{eq:gamma-new} or, equivalently, $u_h^{*}$ by $u_h^{**}$,
%        is to assign the candidate polynomial $u_h^e$ a larger weight in smooth cells such that $C_RR^e<\|u_h-u_h^{*}\|_e$.

We explore the possibility of using residual-based weights to directly construct a consistent WENO approximation $u^{e,*}$ such that $u^{e,*}=u^e$ for a sufficiently smooth exact solution $u^e\in V_h(K_e)$. We found that the difficulty of theoretical studies for WENO schemes partially stems from the lack of this consistency property. The standard definition of nonlinear weights introduces a small consistency error, complicating the analysis of fully stabilized schemes.
	
	Exploiting the local nature of our stabilization technique, we draw upon the HWENO approach proposed by Qiu and Shu \cite{qiu2004} with some modifications. Unlike many HWENO schemes that construct HWENO approximations using only cell averages and first-order derivatives \cite{luo2007,qiu2004}, our approach incorporates all partial derivatives up to order $p$ for polynomial approximation. This eliminates the need for information about neighbors' neighbors, even for higher-order approximations.
	
	Let $K_e\in \mathcal{T}_h$ be a mesh cell and $u_h$ a finite element approximation. We define $\mathcal{S}^e$ as the integer set containing the indices $e'$ of all von Neumann neighboring cells $K_{e'}\in \mathcal{T}_h$. That is, an index $e'$ belongs to $\mathcal{S}^e$ if $K_e$ and $K_{e'}$ have a common boundary (a point in 1D, an edge in 2D, a face in 3D). Since $e\in \mathcal{S}^e$, we can define $\mathcal{S}_0^e:=\{e\}$ and $m_e\geq d+1$ reconstruction stencils $\mathcal{S}_l^e:=\{e,e'\}$, where $e'\in \mathcal{S}^e\setminus\{e\}$.    
		
	A standard HWENO reconstruction on $K_e$ can be written as 
	\begin{align}
	u_h^{e,*} = \sum_{l=0}^{m_e}\omega_l^eu_{h,l}^e\in \mathbb{P}_p(K_e),
	\end{align}
	where $\omega_l^e$ and $u_{h,l}^e$, $l=0,\ldots,m_e$ denote nonlinear weights and Hermite candidate polynomials, respectively. Following Zhong and Shu \cite{zhong2013}, we extend Galerkin polynomials $u_h^{e'}$ of neighboring cells $e'\in \mathcal{S}^e\setminus\{e\}$ into $K_e$. The candidate polynomials are then expressed as
	\begin{align}
		u_{h,l}^e(\mathbf{x})=u_h^{e'}(\mathbf{x})+\pi_e(u_h^e-u_h^{e'}), \qquad \mathbf{x}\in K_e,
	\end{align}
	where $\pi_ev=\frac{1}{|K_e|}\int_{K_e}v\,\mathrm{d}\mathbf{x}$ is the average value of $v\in\{u_h^e,u_h^{e'}\}$ in $K_e$. Equivalently, we have
	\begin{align*}
		\pi_e u_{h,l}^e=\pi_e u_h^e,\qquad D^{\mathbf{k}}u_{h,l}^e=D^{\mathbf{k}}u_h^{e'}, \qquad 1\leq|\mathbf{k}|\leq p.
	\end{align*}
	
	If the approximation $u_h$ varies smoothly on all cells with indices in $S^e$, linear weights $\tilde{\omega}_l^e\in[0,1]$ are assigned to the candidate polynomials $u_{h,l}^e$. Typically, small positive weights $\tilde{\omega}^e_l=10^{-3}$ are assigned to $u_{h,l}^e$, $l=1,\ldots,m_e$, while a large weight $\tilde{\omega}_0^e=1-\sum_{l=1}^{m_e}\tilde{\omega}_l^e$ is assigned to $u_{h,0}^e=u_h^e$ (see, e.g., \cite{jiang1996,zhong2013,zhu2017}). Since our smoothness indicator measures deviations from the WENO reconstruction, assigning a smaller linear weight to $u_{h,0}^e$ results in more reliable shock detection and stronger nonlinear stabilization.
	
	Since the use of linear weights may produce oscillatory reconstructions, nonlinear weights are introduced to measure the relative smoothness of candidate polynomials. The classical smoothness indicators $\beta_l^e$, as proposed by Jiang and Shu \cite{jiang1996} and Friedrich \cite{friedrich1998}, can be written as
	\begin{align*}
		\beta_l^e=\|u_{h,l}^e\|_e^q, \qquad q\geq1,
	\end{align*}
	where $\|\cdot\|_e$ is the semi-norm defined by \eqref{eq:snorm}. We recover the smoothness indicators by Jiang and Shu and Friedrich by using $q=2$ in 1D and $q=1$ in 2D, respectively.
	
	A popular definition of the nonlinear weights $\omega_l^e$ for a WENO scheme is then given by \cite{jiang1996,zhu2009,zhu2017}
	\begin{align}
	\omega_l^e=\frac{\tilde{w}_l^e}{\sum_{k=0}^{m_e}\tilde{w}_k^e}, \qquad \tilde{w}_l^e=\frac{\tilde{\omega}_l^e}{(\varepsilon+\beta_l^e)^r}.
	\label{eq:oldweights}
	\end{align}
	Here, $r$ is a positive integer and $\varepsilon$ is a small positive real number, which is added to avoid division by zero. Typically, the parameters are set to $r=2$ and $\varepsilon=10^{-6}$. Note that this definition of the nonlinear weights introduces a (small) consistency error in the stabilized weak form.
        
	To preserve strong consistency in the sense that $u_h^{e,*}=u_h^e$ for an exact solution $u^e=u|_{K_e}\in V_h(K_e)$, 
        we propose the following new definition of the nonlinear weights:
	\begin{align}
	\omega_l^e=\frac{\tilde{w}_l^e}{\sum_{k=0}^{m_e}\tilde{w}_k^e}, \qquad \tilde{w}_0^e=\frac{\tilde{\omega}_0^e(R^e+\delta)}{(\varepsilon+\beta_0^e)^r}, \qquad
	\tilde{w}_j^e=\frac{\tilde{\omega}_j^e\max\{R^e-\theta R^{e'},0\}}{(\varepsilon+\beta_j^e)^r} \qquad \forall j\in \{1,\ldots,m_e\},
	\label{eq:newweights}
	\end{align}
	where $\delta$ is a small positive number and $R^e$, $R^{e'}$, $e'\in \mathcal{S}^e\setminus \{e\}$ are element residuals ($R^e=\|\mathcal L u_h-g\|_{0,K_e}^2$ for a numerical solution $u_h$ of the localized PDE $\mathcal L u=g$). We set $\delta=\varepsilon=10^{-6}$. Additionally, we introduce a threshold parameter $\theta\in \mathbb{R}_0^+$ to control the sensitivity of $\tilde{w}_j^e$, $j=1,\ldots,m_e$ to the difference between $R^e$ and $R^{e'}$. Essentially, if the residual $R^{e'}$ is large, the nonlinear weight $\tilde{w}_j^e$, $j\in\{1,\ldots,m_e\}$ is set to zero. This approach is akin to the modified smoothness sensor proposed in \cite[Remark 4]{vedral2023} and to the ENO-like stencil selection procedure used in targeted ENO (TENO) schemes \cite{fu2016,fu2017a}.
	
	\begin{remark}
		When $\theta$ is set to 0 and the residuals $R^{e'}$, $e'\in S^e\setminus\{e\}$ are comparable in magnitude, our formulation of the nonlinear weights in \eqref{eq:newweights} closely resembles the conventional definition in \eqref{eq:oldweights}.
	\end{remark}
	
	With the new residual-based definition of the nonlinear weights in \eqref{eq:newweights}, we find that $\omega_0^e=1$ and $\omega_j^e=0$, $j=1,\ldots,m_e$ if $R^e=0$. Consequently, the HWENO reconstruction becomes $u^{e,*}=u^e$, leading to $\|u^{e,*}-u^e\|_{e}=0$ and thus $\gamma_e=1$. Ultimately, we obtain $(1-\gamma_e(u^e))s_h^{e,L}(u^e,w_h)=0$ in \eqref{eq:blend} for all $w_h\in V_h$.
	\medskip
        
	In summary, our arbitrary-order residual-based HWENO scheme ensures strong consistency for DG in the sense that $u^{e,*}=u^e$ and/or $\gamma_e=1$ whenever $R^e=0$. For continuous finite elements the consistency error can be bounded as $\kappa$ contains the difference between the gradients of the finite element solution and their interpolation.

	% ---------------------------------------------------------------------------
	\section{A priori analysis for CG-WENO methods}
	\label{sec:apriori}
	% ---------------------------------------------------------------------------
	Following the theoretical investigations of finite element methods stabilized by shock-capturing \cite{knopp2002,lube2006} and LPS \cite{barrenechea2017c,barrenechea2013} terms, we perform error analysis of dissipation-based WENO stabilization for CG discretizations of scalar convection-dominated transport problems.
	\subsection{Problem statement}
	Let $\Omega \subset \mathbb{R}^d$, $d\in\{1,2,3\}$ be a bounded domain with Lipschitz-continuous boundary $\partial \Omega$. We consider the steady-state convection-diffusion-reaction (CDR) equation
	\begin{align}
	-\varepsilon\Delta u+\mathbf{b}\cdot \nabla u+cu=g \quad \text{in } \Omega, \qquad u=u_D \quad \text{on } \partial\Omega,
	\label{eq:adr}
	\end{align}
	where $\varepsilon>0$ is a constant diffusion coefficient, $\mathbf{b}\in W^{1,\infty}(\Omega)^d$ is a given velocity field%with $\nabla\cdot\mathbf{b}=0$ a.e. in $\Omega$ 
	, $c\in L^{\infty}(\Omega)$ is a non-negative reaction rate, $g\in L^2(\Omega)$ is a given source term, and $u_D \in H^{\frac{1}{2}}(\partial \Omega)$ is a boundary datum. We assume that there exists a constant $\sigma_0$ such that 
	\begin{align}
	\sigma :=c-\frac{1}{2}\nabla\cdot \mathbf{b}\geq \sigma_0 > 0 \quad \text{in } \Omega.
	\label{eq:conduni}
	\end{align}
	This assumption guarantees unique solvability of problem \eqref{eq:adr}.
	
	The weak form of problem \eqref{eq:adr} reads: Find $u\in H^1(\Omega)$ such that $u=u_D$ on $\partial \Omega$ and
	\begin{align}
		a(u,v):=\varepsilon(\nabla u,\nabla v)_\Omega+(\mathbf{b}\cdot \nabla u,v)_\Omega+(cu,v)_\Omega=(g,v)_{\Omega} \qquad \forall v\in H_0^1(\Omega).
		\label{eq:condwf}
	\end{align}
	We introduce a function $u_{h,D}\in V_h$ such that its trace approximates the non-homogeneous Dirichlet boundary condition in \eqref{eq:adr}. The standard CG discretization of problem \eqref{eq:adr} is given by: Find $u_h\in V_h$ such that $u_h-u_{h,D}\in V_{h,0}=V_h\cap H_0^1(\Omega)$ and 
	\begin{align}
	a(u_h,v_h)=(g,v_h)_\Omega \qquad \forall v_h\in V_{h,0}.
	\label{eq:wfg}
	\end{align}
     Let $\Omega_e$ be the patch of elements containing $K_e$ and all elements that share a common vertex with~$K_e$. We define $\omega_e$ as the largest constant less than or equal to one such that the stability estimate 
	\begin{align}
	\|\kappa_e\nabla v\|_{0,K_e}\leq C\|\nabla v\|_{0,\Omega_e} \qquad \forall v\in H^1(\Omega_e)
	\label{eq:szconst}
	\end{align}
        holds for all $C\ge \frac{1}{\sqrt{\omega_e}}$. This constant is used solely in our analysis and does not need to be computed in numerical simulations.

        We can finally state the stabilized semi-discrete formulation of problem \eqref{eq:adr} as follows: Find $u_h\in V_h$ such that $u_h-u_{h,D}\in V_{h,0}$ and
	\begin{align}
		a(u_h,v_h)+s_h(u_h,v_h)&+d_h(u_h;u_h,v_h) =(g,v_h)_{\Omega} \qquad \forall v_h\in V_{h,0}.
		\label{eq:wfs}
	\end{align}
	Here, the bilinear form $s_h(\cdot,\cdot)$ is given by
	\begin{align}
		s_h(w_h,v_h):=\sum_{K_e\in\mathcal{T}_h} \omega_e\nu_e(\kappa_e\nabla v_h,\kappa_e\nabla w_h)_{K_e}\qquad \forall v_h,w_h\in V_h
		\label{eq:def_sh}
	\end{align}
	and the nonlinear form $d_h(\cdot;\cdot,\cdot)$ is given by 
	\begin{align}
	\begin{split}
		d_h(u_h;w_h,v_h) & :=\sum_{K_e\in\mathcal{T}_h}(1-\gamma_e(u_h))\nu_e\Big((\nabla v_h,\nabla w_h)_{\Omega_e}-\omega_e(\kappa_e\nabla v_h,\kappa_e \nabla w_h)_{K_e}\Big)\\
    & \qquad + \sum_{K_e\in\mathcal{T}_h} \gamma_e(u_h) \nu_e \underbrace{ (1 - \omega_e) }_{ \in [0, 1) } (\kappa_e\nabla v_h,\kappa_e \nabla w_h)_{K_e}
		\qquad \forall u_h,v_h,w_h\in V_h.
	\end{split}
	\label{eq:def_dh}
	\end{align}
	The artificial viscosity coefficient \eqref{eq:nuLF} is defined using the maximum speed $\lambda_e=\|\mathbf{b}\|_{L^\infty(K_e)}$. 
   
   	We remark that \eqref{eq:wfs} differs slightly from the nonlinear stabilization presented in \eqref{eq:blend}. The upcoming analysis will show the rationale for these modifications.
     \subsection{Stability}
     We prove the stability of our method \eqref{eq:wfs} w.r.t. the mesh-dependent norm \cite{barrenechea2017c,barrenechea2013}
     \begin{align}
     	\|v\|_S:=(\varepsilon|v|_{1,\Omega}^2+\sigma_0\|v\|_{0,\Omega}^2+s_h(v,v))^{1/2} \qquad \forall v \in H_0^1(\Omega).
     	\label{eq:norm}
     \end{align}
     The following simple but important lemma is the reason for using $s_h(u_h,v_h)+d_h(u_h;u_h,v_h)$ instead of the original stabilization term \eqref{eq:blend} in the method under investigation.
     \begin{lemma}
     	The stabilized scheme \eqref{eq:wfs} is coercive, i.e.,
     	\begin{align}
     		a(v,v)+s_h(v,v)+d_h(w;v,v)\geq \|v\|_S^2 \qquad \forall v\in H_0^1(\Omega),w\in H^1(\Omega),
     		\label{eq:coerc}
     	\end{align}
     	and it is bounded in the sense that there is a constant $\hat{C}>0$ such that
     	\begin{align*}
     		a(v,w)+s_h(v,w)+d_h(u;v,w)\leq \hat{C}\|v\|_S\|w\|_S \qquad \forall u,v,w\in H^1(\Omega).
     	\end{align*}
     	\label{lemma:coerc}
     \end{lemma}
     \begin{proof}
     	The coercivity of linear terms in \eqref{eq:wfs} follows by integration of parts and \eqref{eq:conduni}. Using the stability property \eqref{eq:szconst} of the fluctuation operator $\kappa_e$, we infer that
     	\begin{align*}
     		d_h(w;v,v)\geq Ch\sum_{K_e\in\mathcal{T}_h}\Big(\|\nabla v\|_{\Omega_e}^2-\omega_e\|\kappa_e\nabla v\|_{K_e}^2\Big)\geq 0.
     	\end{align*}
     	Boundedness follows directly from the definition of the bilinear form in \eqref{eq:def_sh} and the fact that $1-\gamma_e(u)\in[0,1]$ for all $u\in H^1(\Omega)$.
     \end{proof}
     
     \subsection{Existence of solutions}
     We proceed by recalling a consequence of Brouwer's fixed point theorem.
     \begin{lemma}
     	Let $X$ be a finite-dimensional Hilbert space with inner product $(\cdot,\cdot)_X$ and norm $\|\cdot\|_X$. Let $M:X\to X$ be a continuous mapping and $K>0$ a real number such that 
     	\begin{align*}
     		(Mx,x)_X>0 \qquad \forall x\in X \; \text{s.t. } \|x\|_X=K.
     	\end{align*}
     	Then there exists at least one $x\in X$ such that $\|x\|_X< K$ and $Mx=0$.
     	\label{lemma:consbrouwer}
     \end{lemma}
     \begin{proof}
     	A proof can be found in \cite[Chapter II, Lemma 1.4]{temam1977}.
     \end{proof}
     The solvability of \eqref{eq:wfs} can be shown using the following theorem.
     \begin{theorem}\label{TH:existence}
     	If the mapping \begin{align}
     		v_h \mapsto \gamma(v_h)
     		\label{eq:condex}
     	\end{align}
     	is continuous for all $K_e\in \mathcal{T}_h$, then there exists a solution $u_h$ of \eqref{eq:wfs}.
     \end{theorem}
     \begin{proof}
     	We assume the continuity of \eqref{eq:condex} for all $K_e\in \mathcal{T}_h$.  Let $M:V_{h,0}\to[V_{h,0}]'$ be defined by 
     	\begin{align*}
     		\big<Mw_h,v_h\big>:=a(w_h+u_{h,D},v_h)+s_h(w_h+u_{h,D},v_h)+d_h(w_h+u_{h,D};w_h+u_{h,D},v_h)-(g,v_h)_\Omega.
     	\end{align*}
     	Setting $w_h=v_h$, we obtain
     	\begin{align*}
  		\begin{split}
   		\big<Mv_h,v_h\big>&=a(v_h,v_h)+s_h(v_h,v_h)+d_h(v_h+u_{h,D};v_h,v_h)\\&+a(u_{h,D},v_h)+s_h(u_{h,D},v_h)+d_h(v_h+u_{h,D};u_{h,D},v_h)-(g,v_h)_\Omega \\
     	&\geq \|v_h\|_S^2-\hat{C}\|u_{h,D}\|_S\|v_h\|_S-\|g\|_{0,\Omega}\|v_h\|_{0,\Omega},
     	\end{split} 
     	\end{align*}
     	where the inequality uses Lemma \ref{lemma:coerc} (twice) and the Cauchy-Schwarz inequality. Since we clearly have that $\|v_h\|_{0,\Omega}\leq \frac{1}{\sigma_0}\|v_h\|_S$, we can further estimate
     	\begin{align}
     		\big<Mv_h,v_h\big>\geq \|v_h\|_S\Bigg(\|v_h\|_S-\hat{C}\|u_{h,D}\|_S-\frac{1}{\sigma_0}\|g\|_\Omega\Bigg),
     	\end{align}
     	which is positive if $v_h$ is chosen so that $\|v_h\|_S$ is large enough. Lemma \ref{lemma:consbrouwer} implies the result.
     \end{proof}
     \begin{remark}
     	Uniqueness of the solution was already discussed in \cite{knopp2002} in the context of stabilized finite element methods with shock capturing. To apply Banach's fixed point theorem, it is necessary to ensure the Lipschitz continuity of \eqref{eq:condex} for a particular choice of the smoothness sensors  $\gamma_e$. However, this requirement appears overly stringent \cite{lube-pc}. An alternative approach could involve proving a uniqueness result for Brouwer's fixed point theorem analogously to the corresponding result for Schauder's fixed point theorem \cite{kellogg1976}. We remark that such an approach also imposes restrictive assumptions on $\gamma_e$ \cite{lube-pc}.
     \end{remark}
    \subsection{Preliminaries}
	We assume that functions belonging to $V_h$ satisfy the local inverse inequality 
	\begin{align}
		|v_h|_{1,K_e}\leq Ch_e^{-1}\|v_h\|_{0,K_e} \qquad \forall v_h\in V_h, K_e\in \mathcal{T}_h.
		\label{eq:invineq}
	\end{align}
	We further suppose that, for a given $h>0$, we have $h_e\leq Ch$ for all $K_e\in \mathcal{T}_h$.
	Let $\mathbf{b}_e\in \mathbb{R}^d$ be a constant vector such that 
	\begin{align}
		|\mathbf{b}_e|\leq \|\mathbf{b}\|_{0,\infty,K_e}\qquad \text{and} \qquad \|\mathbf{b}-\mathbf{b}_e\|_{0,\infty,K_e}\leq Ch_e|\mathbf{b}|_{1,\infty,K_e} \qquad \forall K_e\in \mathcal{T}_h,
		\label{eq:velpre}
	\end{align}
	where $|\cdot|$ denotes the Euclidean norm in $\mathbb{R}^d$. Here, $\mathbf{b}_e$ can be chosen, e.g., as a nodal value of $\mathbf{b}$ in $K_e$ or as the $L^2$ projection of $\mathbf{b}$ into the space of functions that are piecewise constant on $K_e$.

	\subsection{A priori estimates}
	We begin by proving an optimal error estimate for the linear high-order scheme. The derivation of this estimate follows the analysis of LPS approaches in \cite{barrenechea2017c,barrenechea2013}. 
	\begin{theorem}
		Let $u\in H^{k+1}(\Omega)$ for some $k\in\{1,\ldots,p\}$, the mesh family be regular, and let $u_h$ be the solution to \eqref{eq:wfs} without nonlinear stabilization, i.e., set $d_h(u_h;u_h,v_h)=0$ for all $v_h\in V_{h,0}$. Then there is a constant $C>0$, independent of $\varepsilon$ and the mesh size $h$, such that 
		\begin{align}
		\|u-u_h\|_{S}\leq C[\varepsilon^{1/2}+\min(1,\sqrt{h/\varepsilon})]h^k|u|_{k+1,\Omega}.
		\label{eq:estho}
		\end{align}
	\end{theorem}
	\begin{proof}
		Let $I_h:C(\bar{\Omega})\to V_h$ be the standard interpolation operator. We define $\eta:=u-I_hu$ and $\varrho:=u_h-I_hu$ such that $u-u_h=\eta-\varrho$. We have
		\begin{align}
		\|\eta\|_{S}^2\leq\varepsilon|u-I_hu|_{1,\Omega}^2+\sigma_0\|u-I_hu\|_{0,\Omega}^2+s_h(u-I_hu,u-I_hu).
		\label{eq:ie1}
		\end{align}
		The first two terms satisfy
		\begin{align}
		\varepsilon|u-I_hu|_{1,\Omega}^2\leq C\varepsilon h^{2k}|u|_{k+1,\Omega}^2 \quad \text{and} \quad \sigma_0\|u-I_hu\|_{0,\Omega}^2\leq C\sigma_0h^{2k+2}|u|_{k+1,\Omega}^2.
		\label{eq:ie2}
		\end{align}
		Let $\nu^{\max}=\max_{K_e\in\mathcal{T}_h}\nu_e$. Then, we can estimate 
		\begin{align}
			s_h(u-I_hu,u-I_hu)\leq \nu^{\max}\|\kappa(\nabla u - \nabla I_hu)\|_{0,\Omega}^2\leq C \nu^{\max}\|\nabla(u-I_hu)\|_{0,\Omega}^2\leq Ch^{2k+1}|u|_{k+1,\Omega}^2,
			\label{eq:ie3}
		\end{align}
		since the Scott-Zhang operator is $L^2$ stable. Putting \eqref{eq:ie1}, \eqref{eq:ie2} and \eqref{eq:ie3} together, we conclude that 
		\begin{align}
		\|\eta\|_{S}\leq C(\varepsilon^{1/2}+h^{1/2})h^k|u|_{k+1,\Omega}. 
		\label{eq:auxho}
		\end{align}
 By definition of $\|\cdot\|_S$, the contribution of $\varrho$ can be written as
		\begin{align*}
			\|\varrho\|_S&=\frac{a(I_hu-u_h,\varrho)+s_h(I_hu-u_h,\varrho)}{\|\varrho\|_S}\\&=\frac{a(I_hu-u,\varrho)+s_h(I_hu-u,\varrho)}{\|\varrho\|_S}+\frac{a(u-u_h,\varrho)+s_h(u-u_h,\varrho)}{\|\varrho\|_S}=:(I)+(II).
		\end{align*}
		We first derive a bound for (I). We have
		\begin{align*}
		\varepsilon(\nabla(I_hu-u),\nabla v_h)_{\Omega}\leq \varepsilon|I_hu-u|_{1,\Omega}|v_h|_{1,\Omega}\leq C\varepsilon^{1/2}h^k|u|_{k+1,\Omega}\varepsilon^{1/2}|v_h|_{1,\Omega} \leq C\varepsilon^{1/2}h^k|u|_{k+1,\Omega}\|v_h\|_{S}
		\end{align*}
		and 
		\begin{align*}
		(c(I_hu-u),v_h)_{\Omega}\leq c\|I_hu-u\|_{0,\Omega}\|v_h\|_{0,\Omega}\leq Ch^{k+1}|u|_{k+1,\Omega}\|v_h\|_{0,\Omega}\leq Ch^{k+1}|u|_{k+1,\Omega}\|v_h\|_{S}.
		\end{align*}
		The convective part can be estimated by
		\begin{align*}
			(\mathbf{b}\cdot\nabla(I_hu-u_h),v_h)_{\Omega}\leq d \|\mathbf{b}\|_{0,\infty,\Omega}\|\nabla(I_hu-u_h)\|_{0,\Omega}\|v_h\|_{0,\Omega}\leq C h^{k}|u|_{k+1,\Omega}\|v_h\|_{S}.
		\end{align*}
		However, a possibly sharper estimate can be obtained by applying integration by parts such that
		\begin{align*}
		(\mathbf{b}\cdot\nabla(I_hu-u_h),v_h)_{\Omega}=-((\nabla \cdot \mathbf{b})(I_hu-u),v_h)_{\Omega}-(I_hu-u,\mathbf{b}\cdot \nabla v_h)_{\Omega}.
		\end{align*}
		A bound for the first term is given by 
		\begin{align}
		\begin{split}
		((\nabla \cdot \mathbf{b})(I_hu-u),v_h)_{\Omega}\leq \|(\nabla \cdot \mathbf{b})(I_hu-u)\|_{0,\Omega}\|v_h\|_{0,\Omega}\leq Ch^{k+1}|u|_{k+1,\Omega}\|v_h\|_{S}.
		\end{split}
		\label{eq:conv1}
		\end{align}
		The boundedness of the convective field \eqref{eq:velpre} yields
		\begin{align}
			(I_hu-u,\mathbf{b}\cdot \nabla v_h)_{\Omega}\leq d\|\mathbf{b}\|_{0,\infty,\Omega}\|I_hu-u\|_{0,\Omega}\|\nabla v_h\|_{0,\Omega}\leq Ch^{k+1}|u|_{k+1,\Omega}\frac{\|v_h\|_{S}}{\varepsilon^{1/2}}.
		\label{eq:conv2}	
		\end{align} 
		Putting \eqref{eq:conv1} and \eqref{eq:conv2} together, we obtain 
		\begin{align*}
		(\mathbf{b}\cdot\nabla(I_hu-u_h),v_h)_{\Omega}\leq C\frac{1}{\min(1,\varepsilon^{1/2})}h^{k+1}|u|_{k+1,\Omega}\|v_h\|_{S}.
		\end{align*}
		The stabilization term can be estimated by 
		\begin{align*}
			s_h(I_hu-u,\varrho)\leq \nu^{\max}\|\kappa\nabla(I_hu-u)\|_{0,\Omega}\|\kappa\nabla \varrho \|_{0,\Omega}\leq Ch^{k+1}|u|_{k+1,\Omega}\|\varrho\|_S.
		\end{align*}
		Subtracting \eqref{eq:wfs} from \eqref{eq:condwf} yields the approximate Galerkin orthogonality relation
		\begin{align*}
			a(u-u_h,v_h)=s_h(u_h,v_h) \qquad \forall v_h\in V_h.
		\end{align*}
		Thus, $(II)=s_h(u,\varrho)/\|\varrho\|_S$.
		We have
		\begin{align*}
			s_h(u,u)\leq Ch \sum_{K_e\in \mathcal{T}_h}\|\kappa_e\nabla u_h\|_{0,K_e}^2\leq Ch^{2k+1}|u|_{k+1,\Omega}^2
		\end{align*}		
		and finally
		\begin{align*}
			s_h(u,\varrho)\leq \sqrt{s_h(u,u)}\sqrt{s_h(\varrho,\varrho)}\leq Ch^{k+1/2}|u|_{k+1,\Omega}\|\varrho\|_S.
		\end{align*}
		Combining the above auxiliary results proves the validity of \eqref{eq:estho}.
	\end{proof}
	We continue with proving an a priori error estimate for the fully stabilized version \eqref{eq:wfs}.
	\begin{theorem}\label{TH:full_stab}	
	Let $u\in H^{k+1}(\Omega)$ for some $k\in\{1,\ldots,p\}$, the mesh family be regular, and let $u_h$ be the solution to \eqref{eq:wfs}. Then there are constants $C_1,C_2>0$, independent of $\varepsilon$ and the mesh size $h$, such that 
	\begin{align}
		\|u-u_h\|_{S}\leq C_1[\varepsilon^{1/2}+\min(1,\sqrt{h/\varepsilon})]h^k|u|_{k+1,\Omega}+C_2h^{1/2}|u|_{ 1,\Omega}.
	\end{align}
	\end{theorem}
	\begin{proof}
		We follow the proof presented in \cite[Theorem 3.14]{barrenechea2013}.
		
		Let $I_h:C(\bar{\Omega})\to V_h$ be the standard interpolation operator. We define $\eta:=u-I_hu$ and $\varrho:=u_h-I_hu$ such that $u-u_h=\eta-\varrho$. We estimate $\|\eta\|_S$ using \eqref{eq:auxho} and the fact that
		\begin{align*}
			d_h(u_h;u-I_hu,u-I_hu)\leq Ch\sum_{K_e\in\mathcal{T}_h}\|\nabla(u-I_hu)\|_{0,K_e}^2\leq Ch^{2k+1}|u|_{k+1,\Omega}^2.
		\end{align*}By \eqref{eq:condwf} and \eqref{eq:wfs}, we have
		\begin{align*}
			&a(\varrho,\varrho)+s_h(\varrho,\varrho)+d_h(u_h;u_h,\varrho)\\&=a(u_h,\varrho)+s_h(u_h,\varrho)+d_h(u_h;u_h,\varrho)-a(I_hu,\varrho)-s_h(I_hu,\varrho)\\&=a(\eta,\varrho)+s_h(\eta,\varrho)-s_h(u,\varrho).
		\end{align*}
		Using definition \eqref{eq:norm} of $\|\cdot\|_S$ and \eqref{eq:coerc}, we obtain
		\begin{align*}
			&\|\varrho\|_S^2+d_h(u_h;\varrho,\varrho)\\&\leq a(u_h,\varrho)+s_h(u_h,\varrho)+d_h(u_h;u_h,\varrho)-a(I_hu,\varrho)-s_h(I_hu,\varrho)-d_h(u_h;I_hu,\varrho)\\&=a(u,\varrho)-a(I_hu,\varrho)-s_h(I_hu,\varrho)-d_h(u_h;I_hu,\varrho)\\&=a(\eta,\varrho)+s_h(\eta,\varrho)-s_h(u,\varrho)-d_h(u_h;I_hu,\varrho).
		\end{align*}
		The first three terms can be estimated using \eqref{eq:estho}. To bound the nonlinear term, we use the coercivity of the nonlinear form in \eqref{eq:coerc}. The application of Hölder's and Young's inequalities yields
		\begin{align*}
			d_h(u_h;I_hu,\varrho)\leq\sqrt{d_h(u_h;I_hu,I_hu)}\sqrt{d_h(u_h;\varrho,\varrho)}\leq d_h(u_h;I_hu,I_hu)+\frac{1}{4}d_h(u_h;\varrho,\varrho),
		\end{align*}
where 
		\begin{align}
		\begin{split}
			d_h(u_h;I_hu,I_hu)\leq \sum_{K_e\in\mathcal{T}_h}(1-\gamma_e(u_h))\nu_e\|\nabla I_hu\|_{0,K_e}^2\leq Ch\sum_{K_e\in \mathcal{T}_h}\|\nabla I_hu\|_{0,K_e}^2\leq Ch|u|_{1,\Omega}^2.
			%&\leq \sum_{K_e\in\mathcal{T}_h}(1-\gamma_e(u_h))\nu_e\|\nabla I_hu\|_{0,K_e}^2\leq Ch\sum_{K_e\in\mathcal{T}_h}1-\gamma_e(u_h)\\&\leq Ch\sum_{K_e\in\mathcal{T}_h}\left(\min\Bigg\{1, \frac{\min\{C_RR^e,\|u_h-u_h^{*}\|_e\}}{\|u_h\|_e}\Bigg\}\right)^q\leq Ch\sum_{K_e\in\mathcal{T}_h}(R^e)^q.
		\end{split}
		\label{eq:nonlest}
		\end{align}
                It follows that
		\begin{align*}
			\|\varrho\|_S^2+d_h(u_h;\varrho,\varrho)\leq C_1^2[\varepsilon^{1/2}+\min(1,\sqrt{h/\varepsilon})]^2h^{2k}|u|_{k+1,\Omega}^2+C_2^2h|u|_{1,\Omega}^2,
		\end{align*}
		which completes the proof.
	\end{proof}
	\begin{remark}\label{REM:higher_conv}
		In \eqref{eq:nonlest}, it was shown that $\sqrt{d_h(u_h;I_hu,I_hu)}=\mathcal{O}(h^{1/2})$ in the worst case. However, this a priori estimate is not sharp as it neglects the actual dependence of the blending factor $1-\gamma_e(u_h)$ on the mesh size $h$. In fact, assuming that
		\begin{align*}
			\|u_h-u_h^*\|_e\leq Ch^{2k/q}\|u_h\|_e
		\end{align*}
                for all $K_e\in\mathcal T_h$,
		we obtain the optimal a posteriori error estimate $\sqrt{d_h(u_h;I_hu,I_hu)}=\mathcal{O}(h^{k+1/2})$.
	\end{remark}

\renewcommand{\vec}{\mathbf}
\newcommand*{\avg}[1]{\ensuremath{\{\![ #1 ]\!\}}}
\newcommand*{\jump}[1]{\ensuremath{[\![ #1 ]\!]}}
\newcommand*{\upw}[1]{\ensuremath{[ #1 ]^\uparrow}}
\newcommand*{\ds}{\ensuremath{\; \textup d\sigma}}
\newcommand*{\dx}{\ensuremath{\; \textup d \vec x}}
\section{A priori analysis for DG-WENO methods}
\label{sec:aprioriDG}
Let us transfer the results of Section \ref{sec:apriori} to the discontinuous Galerkin (DG) setting, in which they, in fact, hold with even sharper bounds. To that end, we approximate the solution $u \in H^2(\Omega)$ of \eqref{eq:adr} subject to \eqref{eq:conduni} by the broken polynomial $u_h \in V_h^d(\mathcal T_h)$ that satisfies
\begin{equation}\label{EQ:dg_scheme}
 a(u_h, v_h) + d_h(u_h; u_h, v_h) = \ell_h(v_h)
\end{equation}
for all $v_h \in V_h^d(\mathcal T_h)$. Here, the DG bilinear form $a_h$ is defined as
\begin{align*}
 a(v_h, w_h) & := \varepsilon \left[ \sum_{K_e \in \mathcal T_h} (\nabla w_h, \nabla v_h)_{K_e} - \sum_{F \in \mathcal F} \int_F \avg{\nabla w_h} \cdot \jump{v_h} + \avg{\nabla v_h} \cdot \jump{w_h} - \frac{\eta}{h_F} \jump{w_h} \cdot \jump{v_h} \ds \right] \\
 & \qquad - \sum_{K_e \in \mathcal T_h} ( v_h, \vec b \cdot \nabla w_h )_{K_e} + \sum_{F \in \mathcal F} \int_F \upw{\vec b v_h} w_h \ds + \left((c - \nabla \cdot \vec b) v_h, w_h \right)_\Omega,
\end{align*}
where $\mathcal F$ denotes the set of faces of the mesh $\mathcal T_h$ and $\eta > 0$ is a sufficiently large stabilization parameter. For faces $F \in \mathcal F$ belonging to two adjacent mesh elements $K^+, K^- \in \mathcal T_h$, the averaging, jump, and upwind operators take the form
\begin{gather*}
 \avg{\nabla v_h}|_F = \tfrac12 (\nabla v_h|_{K^+} + \nabla v_h|_{K^-}), \qquad \qquad \jump{ v_h }|_F = v_h|_{K^+} \vec n_{K^+} + v_h|_{K^-} \vec n_{K^-},\\
 \upw{\vec b v_h}(x) = \begin{cases} (\vec b \cdot \vec n_{K^+}) v_h|_{K^+} & \text{ if } \vec b \cdot \vec n_{K^+} \ge 0, \\ (\vec b \cdot \vec n_{K^-}) v_h|_{K^-} & \text{ otherwise.}
 \end{cases}
\end{gather*}
If the face $F \subset \partial \Omega$ is adjacent to $K^+ \in \Omega$, the above definitions still hold if the expression $v_h|_{K^-}$ is replaced by $v_h|_{K^+}$ in the average and by zero in the jump. The stabilization operator $d_h$ is defined as in \eqref{eq:blend}. Unlike the CG case, no modifications to the nonlinear stabilization are necessary. However, since we replace the Scott--Zhang interpolation by the $L^2$ orthogonal projection onto $V^d_h(\mathcal T_h)$ in the definition of the fluctuation operator \eqref{eq:fluctuation}, we can immediately conclude that $\kappa (\nabla v_h) = 0$ for all $v_h \in V_h^d(\mathcal T_h)$, and thus
\begin{equation*}
 d_h(u_h; v_h, w_h) = \sum_{K_e \in \mathcal T_h} (1 - \gamma_e(u_h)) \nu_e \int_{K_e} \nabla w_h \cdot \nabla v_h \dx.
\end{equation*}
Finally, the linear form $\ell_h$ consistently incorporates the right-hand side $g$ and boundary data $u_\textup D$. Let us define the DG counterpart of the $\| \cdot \|_S$ norm via
\begin{equation*}
 \| v_h \|^2_\textup{DG} := \varepsilon \left[ \sum_{K_e \in \mathcal T_h} \| \nabla v_h \|^2_{0,K_e} + \sum_{F \in \mathcal F} \tfrac1{h_F} \jump{v_h} |^2_F \right] + \sigma_0 \| v_h \|^2_{0,\Omega}.
\end{equation*}

\begin{lemma}\label{LEM:disc_coercivity}
 Assume that $\eta > 0$ is sufficiently large (a precise criterion is given in \cite[Lemma\ 4.12]{di-pietro2012}). Then there is a constant $C_\eta$, independent of $\varepsilon$ and the mesh, such that
 \begin{equation*}
  a(v_h, v_h) + d_h(w_h; v_h, v_h) \ge C_\eta \| v_h \|^2_\textup{DG} \qquad \forall v_h, w_h \in V_h^d(\mathcal T_h).
 \end{equation*}
\end{lemma}
\begin{proof}
 We have $0 \le d_h(w_h; v_h, v_h)$. Thus, the result is a straightforward combination of \cite[Lemma\ 4.20]{di-pietro2012} for the diffusive part and \cite[Lemma\ 2.27]{di-pietro2012} for the advection--reaction part.
\end{proof}

\begin{lemma}\label{LEM:dg_existence}
 If the mapping $v_h \mapsto \gamma(v_h)$ is continuous and $\eta > 0$ is chosen as in Lemma \ref{LEM:disc_coercivity}, then there exists a solution to \eqref{EQ:dg_scheme}.
\end{lemma}
\begin{proof}
 As in the proof of Theorem \ref{TH:existence}, we define $M \colon V_h^d(\mathcal T_h) \to V_h^d(\mathcal T_h)$ via 
 \begin{equation*}
  \langle M v_h, w_h \rangle := a(v_h, w_h) + d_h(v_h; v_h, w_h) - \ell_h(v_h) \qquad \forall v_h, w_h \in V_h^d(\mathcal T_h).
 \end{equation*}
 This directly implies that
 \begin{equation*}
  \langle M v_h, v_h \rangle \ge C_\eta \| v_h \|^2_\textup{DG} - \| \ell_h \| \| v_h \|_\textup{DG},
 \end{equation*}
 where $\| \ell_h \|$ is the operator norm of the right-hand side's linear form, which is bounded since we consider finite-dimensional vector spaces. The arguments in the proof of Theorem \ref{TH:existence} guarantee the existence of $u_h \in V_h^d(\mathcal T_h)$ that satisfies \eqref{EQ:dg_scheme}.
\end{proof}

\begin{lemma}
Let the assumptions of Lemma \ref{LEM:dg_existence} hold. Suppose that the solution $u$ of the continuous problem has $H^{k+1}(K_e)$, $k\in\{1,\ldots,p\}$ regularity for all $K_e \in \mathcal T_h$. Then there are constants $C_1,C_2>0$, independent of $\varepsilon$ and the mesh, such that
 \begin{equation*}
  \| u - u_h \|_\textup{DG} \le C_1 h^k \sqrt{\varepsilon + h} | u |_{k+1,\mathcal T_h} + C_2 h^{1/2} | u |_{1,\mathcal T_h}.
 \end{equation*}
\end{lemma}
Here, the broken norm is defined as $| v_h |^2_{k,\mathcal T_h} = \sum_{K_e \in \mathcal T_h} | v_h |^2_{k, K_e}$.
\begin{proof}
 The result clearly holds for the scheme without stabilization, i.e., for $u_h$ satisfying \eqref{EQ:dg_scheme} without the term $d_h$. In this case, the validity of the claim can be deduced using a simple combination of the arguments in \cite[Sect.\ 2.3.2 and 4.2.3.3]{di-pietro2012}. The full result follows from an application of Strang's first lemma as in the proof of Theorem \ref{TH:full_stab}.
\end{proof}
% --------------------------------------------------------------------------------

% ---------------------------------------------------------------------------
\section{Numerical examples}
\label{sec:num}
% ---------------------------------------------------------------------------
In this section, we present numerical studies of the residual-based WENO (RB-WENO) scheme for both linear and nonlinear scalar problems, as well as for the Euler equations of gas dynamics. We demonstrate the excellent shock-capturing capabilities of the RB-WENO scheme for continuous and discontinuous finite elements, and its ability to converge to correct entropy solutions in the nonlinear case. To compare our results with those in \cite{kuzmin2023a,vedral2023}, we use continuous finite elements for scalar benchmarks and discontinuous finite elements for the Euler equations. 

Throughout all examples, we employ Lagrange basis functions of order $p$. Time integration is performed using strong stability preserving (SSP) Runge-Kutta methods of order $p+1$ \cite{gottlieb2001}. Following the approaches presented in \cite{kuzmin2023a,vedral2023}, we set $q=1$ in \eqref{eq:gamma-old}.

For clarity, we label the schemes from \cite{kuzmin2023a,vedral2023} as CG-WENO and DG-WENO for continuous and discontinuous finite elements, respectively. Our newly proposed scheme is referred to as CG-RB-WENO-$\theta$ and DG-RB-WENO-$\theta$, where $\theta$ is the sensitivity parameter introduced in \eqref{eq:newweights}.
All schemes are implemented using the open-source C++ library MFEM \cite{anderson2021,mfem}, and the results of two-dimensional test problems are visualized using the open-source C++ software GLVis \cite{glvis}.
% and deviate from the default setting of the linear weights in, e.g., \cite{kuzmin2023a,shu1989}, by setting $\tilde{\omega}_0^e=1-m_e\cdot10^{-2}$, $\tilde{\omega}_l^e=10^{-2}$, $l=1,\ldots,m_e$ for discontinuous finite elements (cf. \cite{vedral2023}).
\subsection{Solid body rotation}
We begin by considering LeVeque's \cite{leveque1996} solid body rotation problem which is a popular stability test for discretizations of the transport problem 
\begin{align*}
\frac{\partial u}{\partial t}+\nabla\cdot(\mathbf{v}u) = 0 \qquad \text{in} \; \Omega=(0,1)^2.
\end{align*}
The divergence-free velocity field ${\bf v}(x,y)=2\pi(0.5-y,x-0.5)$ rotates a smooth hump, a sharp cone and a slotted cylinder around the center of the domain. After each revolution ($t=2\pi r$, $r\in \mathbb{N}$) the exact solution coincides with the initial condition given by
\begin{equation}
u_0(x,y) = \begin{cases}
u_0^{\text{hump}}(x,y) & \text{if } \sqrt{(x-0.25)^2+(y-0.5)^2}\le 0.15, \\
u_0^{\text{cone}}(x,y) & \text{if } \sqrt{(x-0.5)^2+(y-0.25)^2}\le 0.15, \\
1 & \text{if }\begin{cases}
\sqrt{(x-0.5)^2+(y-0.75)^2}\le0.15, \\
|x-0.5|\ge0.025, y\ge0.85,
\end{cases}\\
0 & \text{otherwise},
\end{cases}
\end{equation}
where
\begin{align*}
u_0^{\text{hump}}(x,y) &= \frac{1}{4}+\frac{1}{4}\cos\bigg(\frac{\pi\sqrt{(x-0.25)^2+(y-0.5)^2}}{0.15}\bigg), \\
u_0^{\text{cone}}(x,y) &= 1-\frac{\sqrt{(x-0.5)^2+(y-0.25)^2}}{0.15}.
\end{align*}

We evolve numerical solutions up to the final time $t=1.0$ on a uniform quadrilateral mesh using $E_h=128^2$ elements and $p=2$. The results are shown in Figs \ref{fig:sbr1}-\ref{fig:sbr3}. Consistent with \cite{kuzmin2023a}, the CG-WENO scheme preserves the structure of all rotating objects but exhibits a 'dip' at the back of the slotted cylinder. This issue can be resolved by employing the CG-RB-WENO scheme, even with a small parameter $\theta\leq 1$, at the cost of losing bound preservation. If necessary, global bounds can be preserved using the flux limiting techniques presented in, e.g., \cite{kuzmin2023}.
%Is there a better way to do it?
\captionsetup[subfigure]{% use subfigure to confine changes to subcaption
	format = hang,
	justification = centering, % Or justified
}
\begin{figure}[!htb]
	\centering
	\begin{subfigure}[b]{.32\linewidth}
		\includegraphics[width=\linewidth]{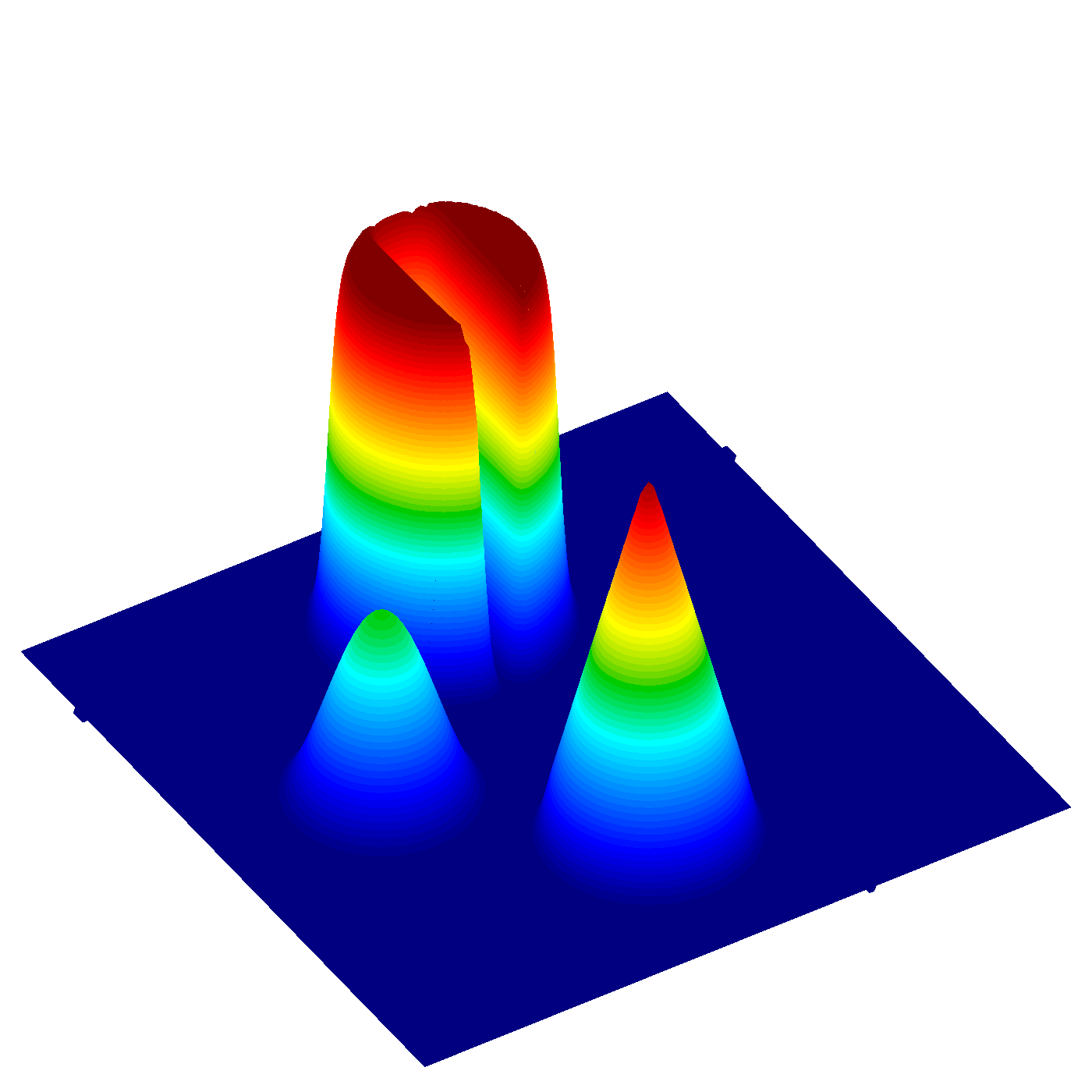}
		\caption{CG-WENO,\newline\hspace*{0.cm} $u_h \in [0.000,1.000]$}
		\label{fig:sbr1}
	\end{subfigure}
	\begin{subfigure}[b]{.32\linewidth}
		\includegraphics[width=\linewidth]{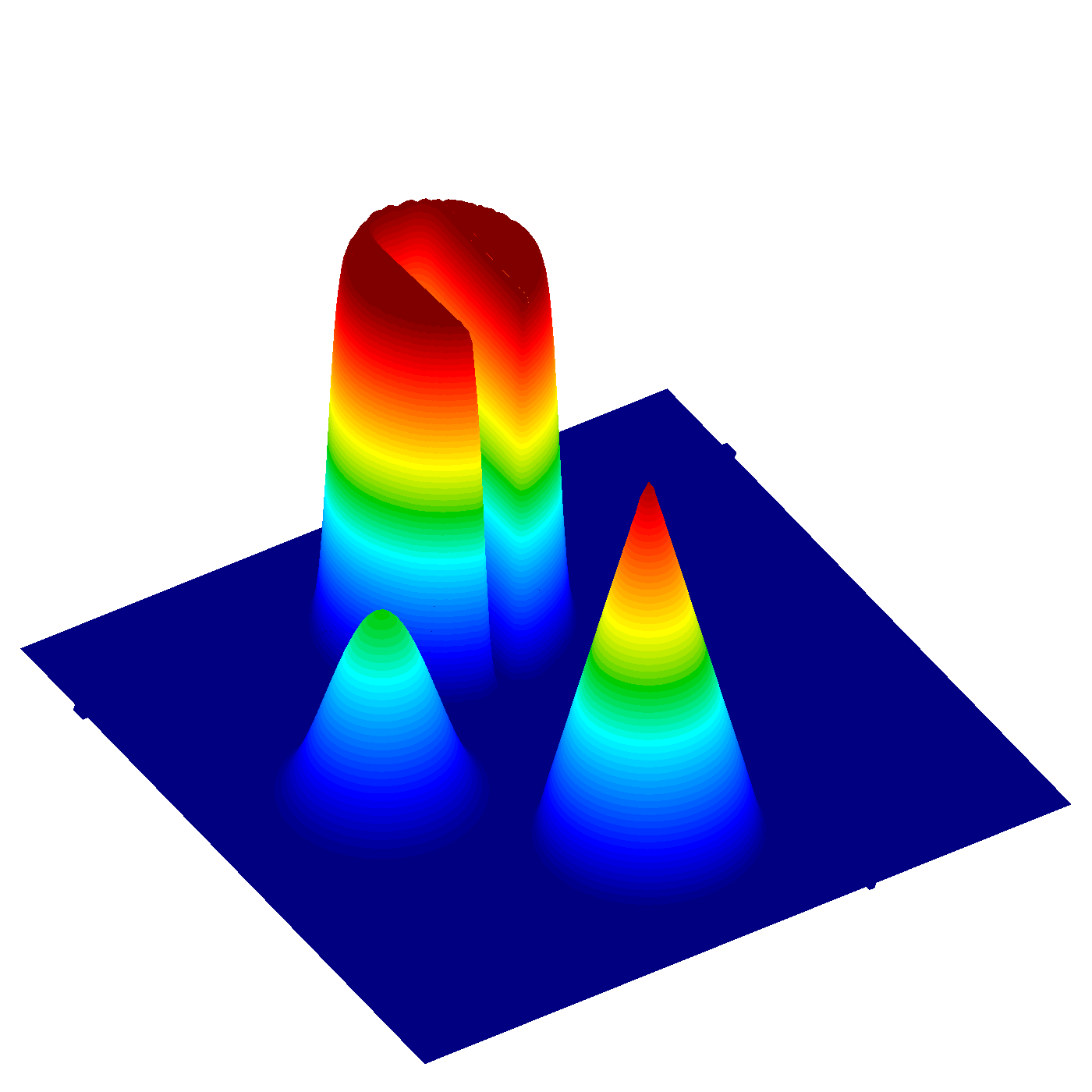}
		\caption{CG-RB-WENO-$0.1$, $u_h \in [-0.008,1.010]$}
		\label{fig:sbr2}
	\end{subfigure}
	\begin{subfigure}[b]{.32\linewidth}
		\includegraphics[width=\linewidth]{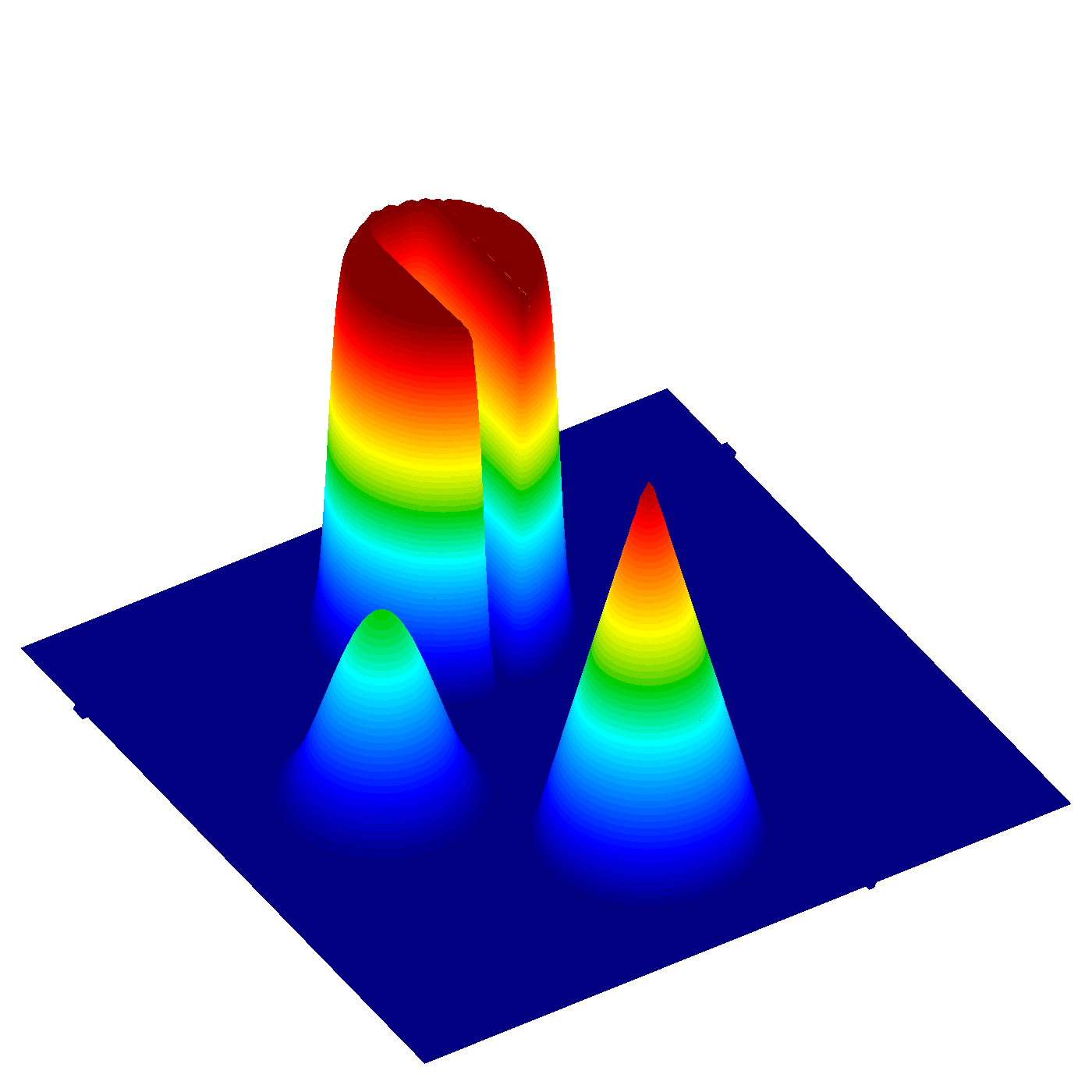}
		\caption{CG-RB-WENO-$1.0$, $u_h \in [-0.009,1.010]$}
		\label{fig:sbr3}
	\end{subfigure}
	\caption{Solid body rotation, numerical solutions at $t=1.0$ obtained using $E_h=128^2$ and $p=2$.}
	\label{fig:sbr}
\end{figure}
\subsection{KPP problem}
To assess the entropy stability properties of our scheme, we consider the two-dimensional KPP problem \cite{kurganov2007}. In this example, we solve 
\begin{align*}
\frac{\partial u}{\partial t}+\nabla \cdot \mathbf{f}(u)=0.
\end{align*}
The nonconvex flux function 
\begin{align*}
\mathbf{f}(u)=(\sin(u),\cos(u))
\end{align*}
produces an entropy solution displaying a rotational wave structure. Many stabilized DG schemes require an additional entropy fix \cite{kuzmin2021a,kuzmin2023,moujaes2023} to avoid potential convergence to incorrect weak solutions.
The computational domain is $\Omega=(-2,2)\times(-2.5,1.5)$ on which the initial condition is given by
\begin{align*}
u_0(x,y)=\begin{cases}
\frac{7\pi}{2} & \text{if }\sqrt{x^2+y^2}\le 1,\\
\frac{\pi}{4} & \text{otherwise}.
\end{cases}
\end{align*}

We use $\lambda_e=1.0$ as an upper global bound for the maximum wave speed to compute the viscosity parameter $\nu_e$ in \eqref{eq:nuLF}. We refer the reader to \cite{guermond2017} for more accurate bounds. We deviate from the default settings of our HWENO scheme as described in \cite{kuzmin2023a} by setting the linear weights $\tilde{\omega}_l^e=0.2$, $l=0,\ldots,4$.

As in the previous example, we perform numerical simulations on a uniform quadrilateral mesh using $E_h=128^2$ elements. The results for quadratic finite elements are displayed in Figs \ref{fig:kpp1}-\ref{fig:kpp3}. The CG-WENO scheme successfully converges to the correct entropy solution. Despite being less diffusive due to the ENO-like stencil selection near discontinuities, our CG-RB-WENO scheme also converges to the correct entropy solution without any additional entropy fixes.
\begin{figure}[!htb]
	\centering
	\begin{subfigure}[b]{.32\linewidth}
		\includegraphics[width=\linewidth]{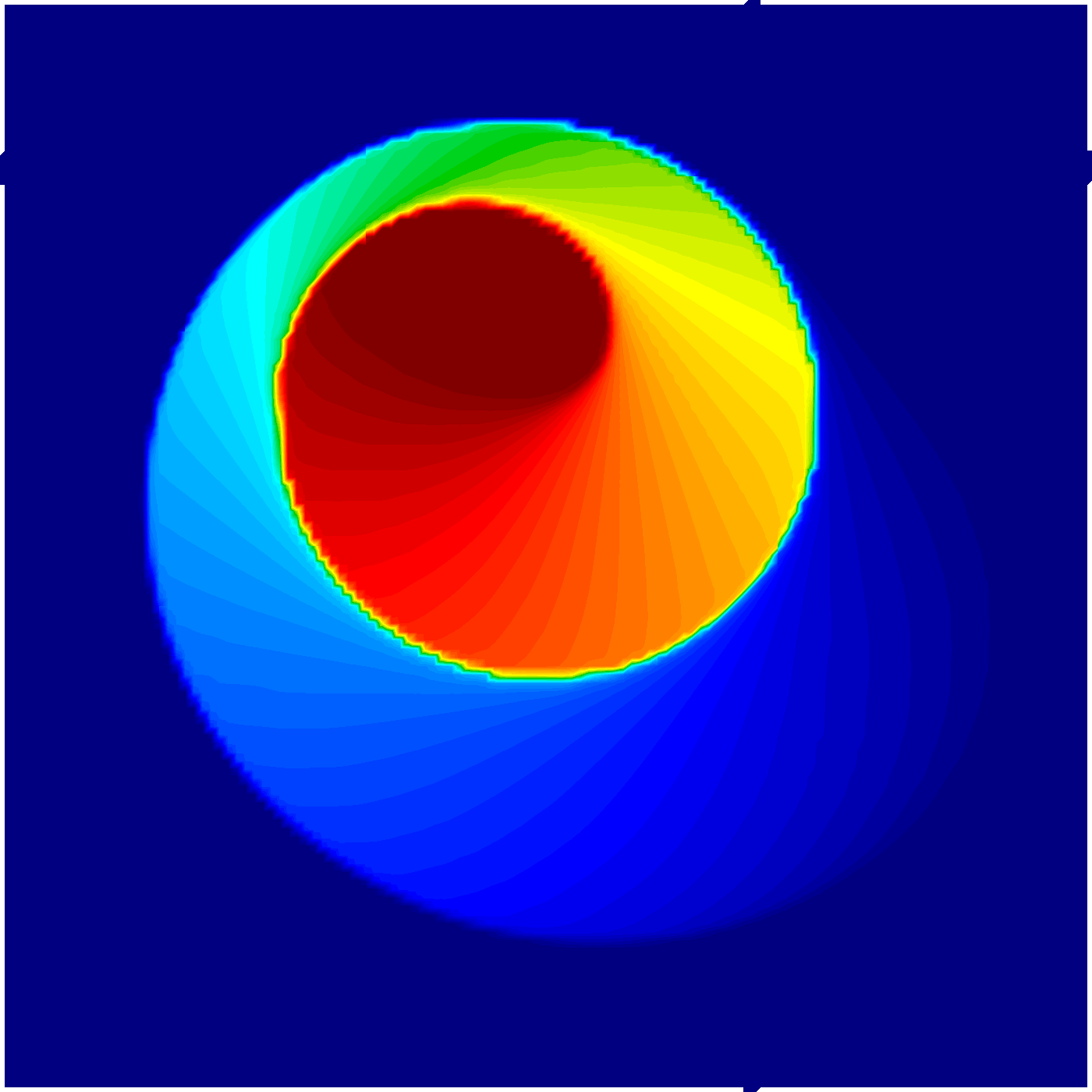}
		\caption{CG-WENO,\newline\hspace*{0.4cm} $u_h \in [0.7777,11.076]$}
		\label{fig:kpp1}
	\end{subfigure}
	\begin{subfigure}[b]{.32\linewidth}
		\includegraphics[width=\linewidth]{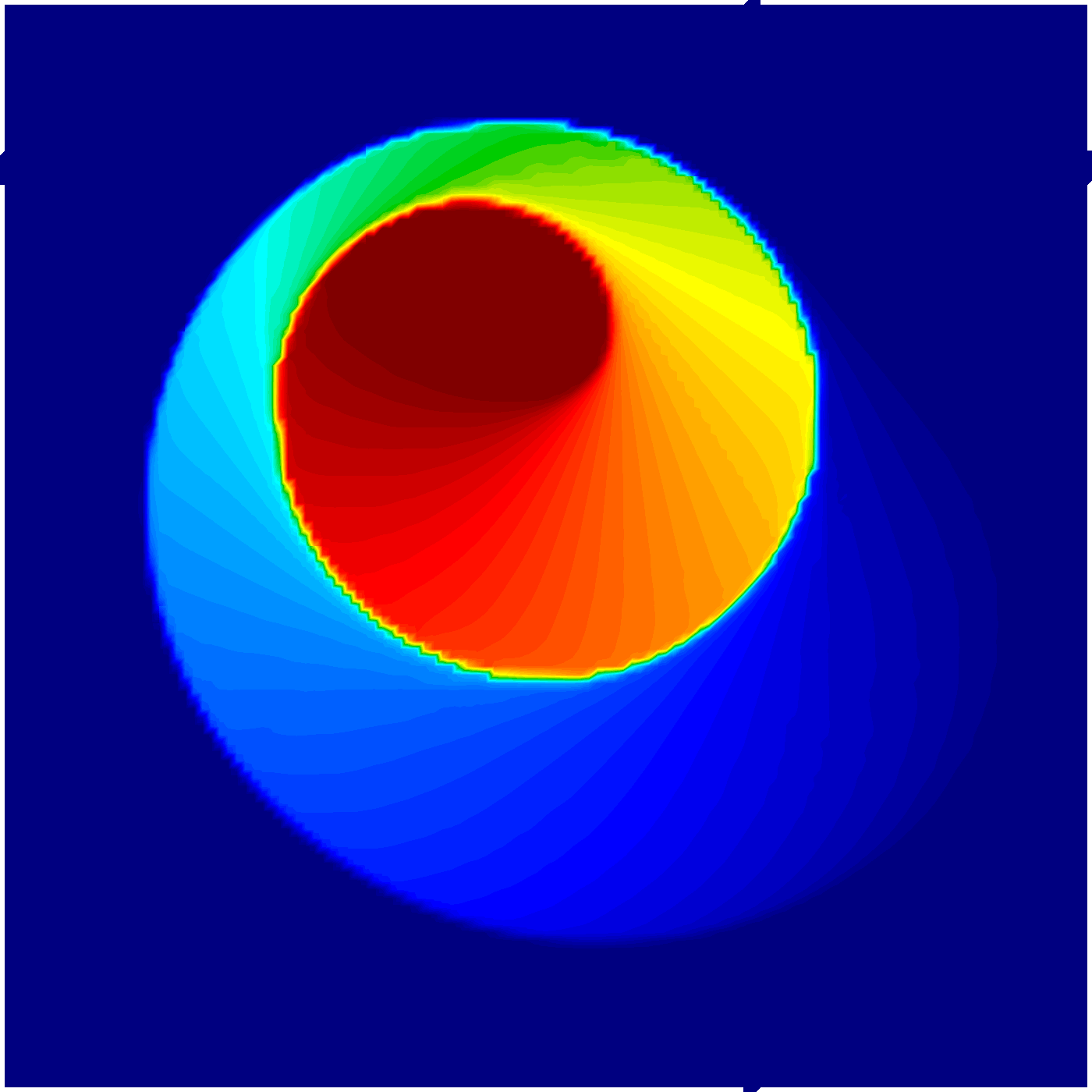}
		\caption{CG-RB-WENO-$0.1$, $u_h \in [0.743,11.048]$}
		\label{fig:kpp2}
	\end{subfigure}
	\begin{subfigure}[b]{.32\linewidth}
		\includegraphics[width=\linewidth]{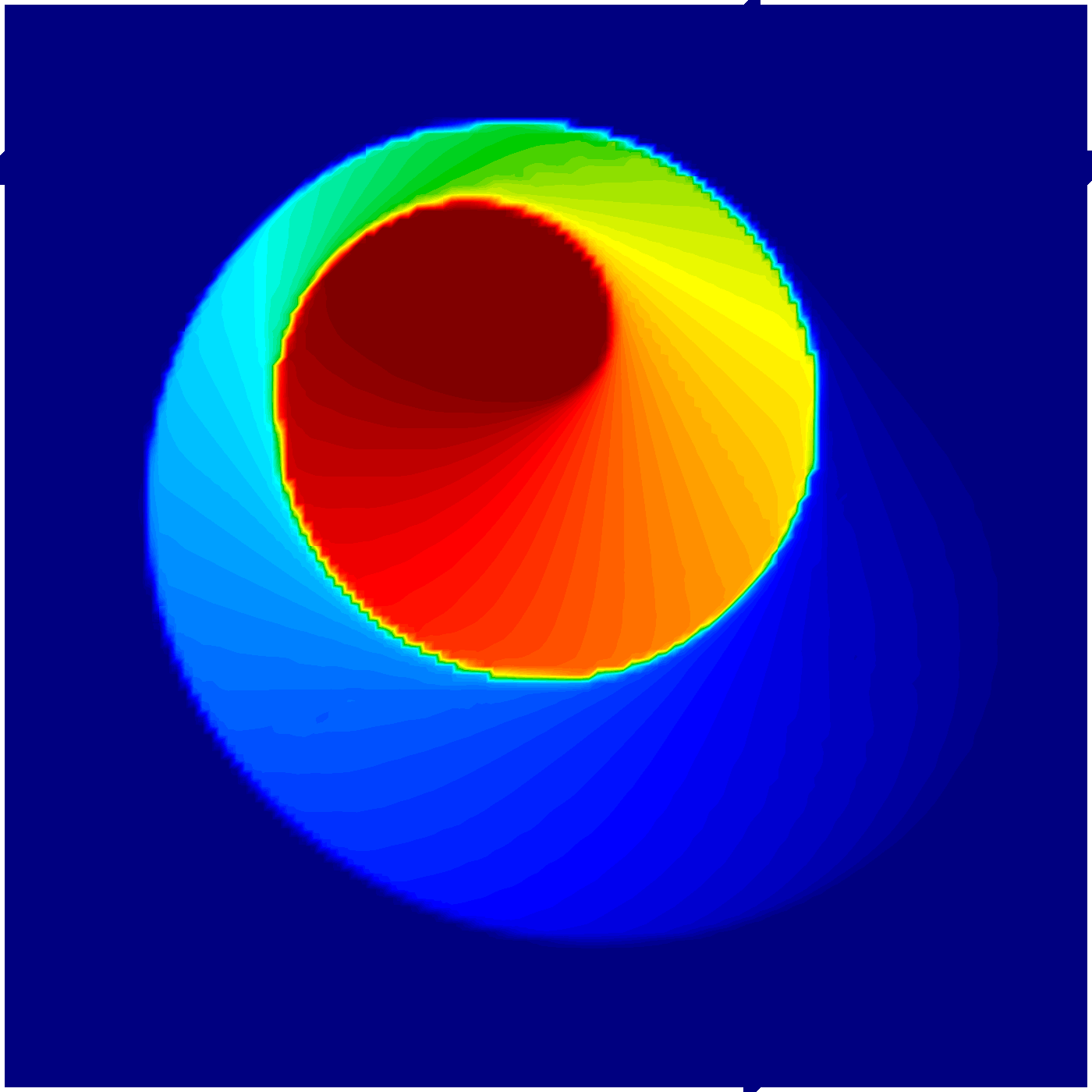}
		\caption{CG-RB-WENO-$1.0$, $u_h \in [0.740,11.044]$}
		\label{fig:kpp3}
	\end{subfigure}
	\caption{KPP problem, numerical solutions at $t=1.0$ obtained using $E_h=128^2$ and $p=2$.}
	\label{fig:kpp}
\end{figure}

\subsection{Euler equations of gas dynamics}
We consider the Euler equations of gas dynamics which represent the conservation of mass, momentum and total energy. The solution vector and flux matrix read
\begin{align*}
U = \begin{bmatrix}
\varrho \\
\varrho \mathbf{v}\\
\varrho E
\end{bmatrix}\in \mathbb{R}^{d+2}, \quad 
\mathbf{F(U)}=\begin{bmatrix}
\varrho \mathbf{v} \\
\varrho \mathbf{v} \bigotimes \mathbf{v}+p\mathbf{I}\\
(\varrho E + p)\mathbf{v}
\end{bmatrix}\in \mathbb{R}^{(d+2) \times d}.
\end{align*}
Here, $\varrho$, $\mathbf{v}$, $E$ denote the density, velocity and specific total energy, respectively. We further denote by $\mathbf{I}$ the identity matrix and by $p$ the pressure which is computed using the polytropic ideal gas equation
$$
p=\varrho e (\gamma - 1),
$$
where $\varrho e$ and $\gamma$ denote the internal energy and the heat capacity ratio, respectively. We set $\gamma=1.4$ throughout all numerical simulations.

To reduce the computational effort, we determine our blending factor using only information from the density field and apply this factor to both the momentum and energy computations (cf. \cite{vedral2023,zhao2020}).

\subsubsection{Titarev-Toro problem}
We consider the Titarev-Toro problem \cite{titarev2004} which is a more complex variant of the sine-shock interaction problem, also known as the Shu-Osher problem \cite{shu1989}. Similarly to the Shu-Osher setup, the computational domain $\Omega=(-5,5)$ is bounded by an inlet on the left boundary and a reflecting wall on the right boundary. This test is widely used to evaluate a scheme's ability to accurately represent high-frequency waves behind a shock. The problem is equipped with the initial conditions
\begin{align*}
	\begin{bmatrix}
	\varrho_L \\v_L\\p_L
	\end{bmatrix}=
	\begin{bmatrix}1.515695\\0.523346\\1.805
	\end{bmatrix}, \quad 
	\begin{bmatrix}
	\varrho_R \\v_R\\p_R
	\end{bmatrix}=
	\begin{bmatrix}1.0+0.1\sin(20\pi(x-5.0))\\0.0\\1.0
	\end{bmatrix}.
\end{align*}  
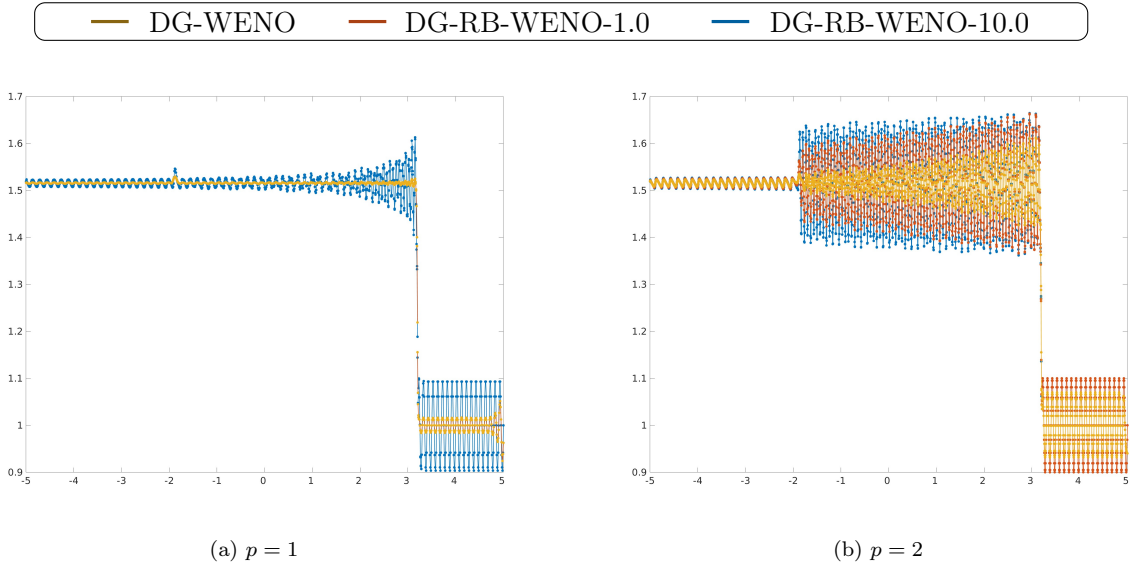
\begin{figure}[!htb]
	\centering
	\begin{subfigure}[b]{\linewidth}
		\centering
		\begin{tikzpicture}
		\draw[rounded corners] (0, 0) rectangle (14, 0.5) node[pos=.5]{};
		\draw[very thick, color={rgb:red,0;green,0.4470;blue,0.7410}] (9,0.25)--(9.5,0.25);
		%(2.0,0.25) node [behind path] {DG-WENO};
		\node at (2.5,0.25) (a) {DG-WENO};
		\node at (6.55,0.25) (a) {DG-RB-WENO-$1.0$};
		\node at (11.5,0.25) (a) {DG-RB-WENO-$10.0$};
		\draw[very thick,color={rgb:red,0.9290;green,0.6940;blue,0.1250}] (0.75,0.25)--(1.25,0.25);
		\draw[very thick,color={rgb:red,0.8500;green,0.3250;blue,0.0980}] (4.175,0.25)--(4.675,0.25);
		\end{tikzpicture}
		\vspace*{0.25cm}
	\end{subfigure}
	\begin{subfigure}[b]{.48\linewidth}
		\includegraphics[width=\linewidth]{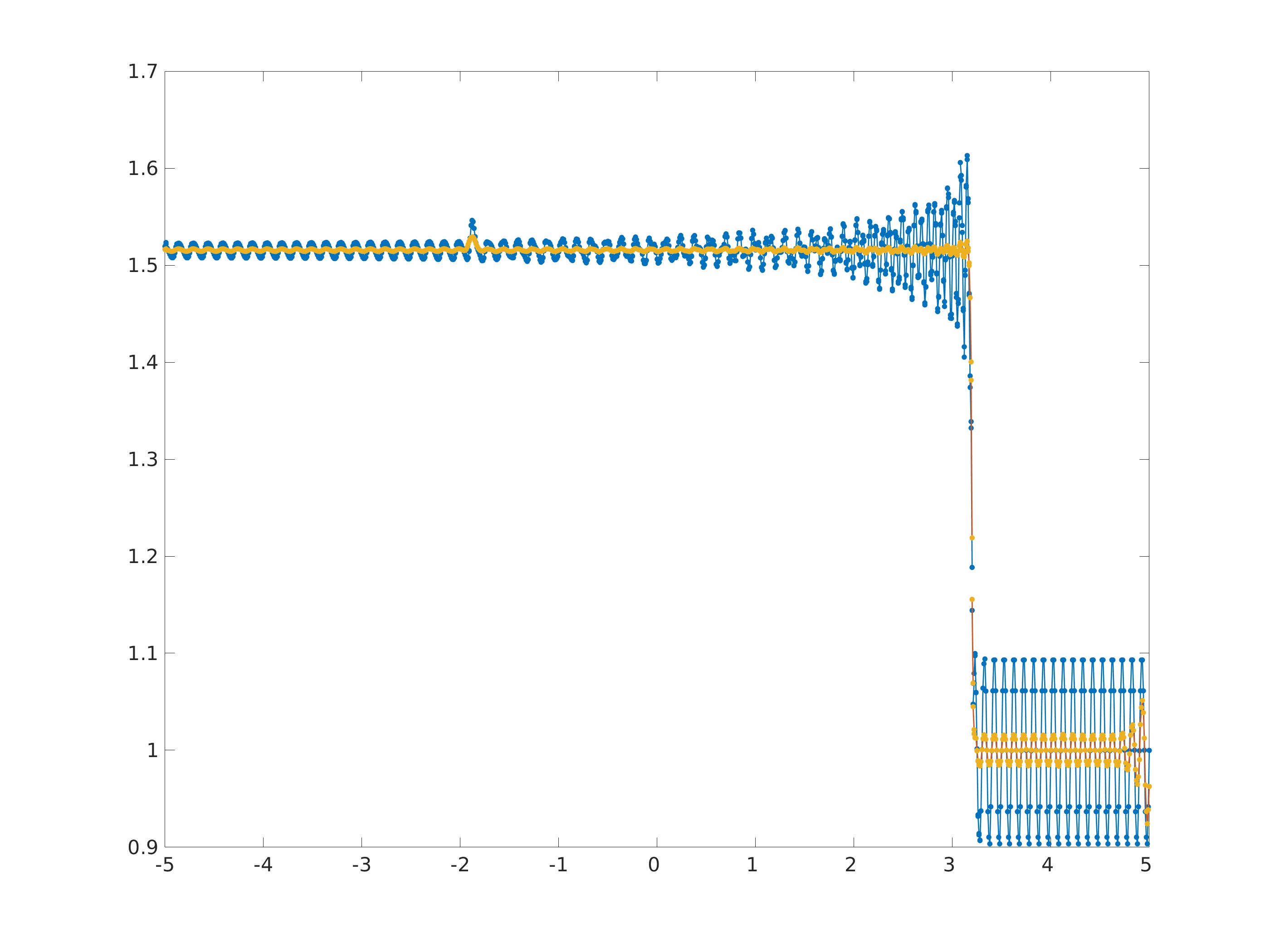}
		\caption{$p=1$}
		\label{fig:ttp1}
	\end{subfigure}
	\begin{subfigure}[b]{.48\linewidth}
		\includegraphics[width=\linewidth]{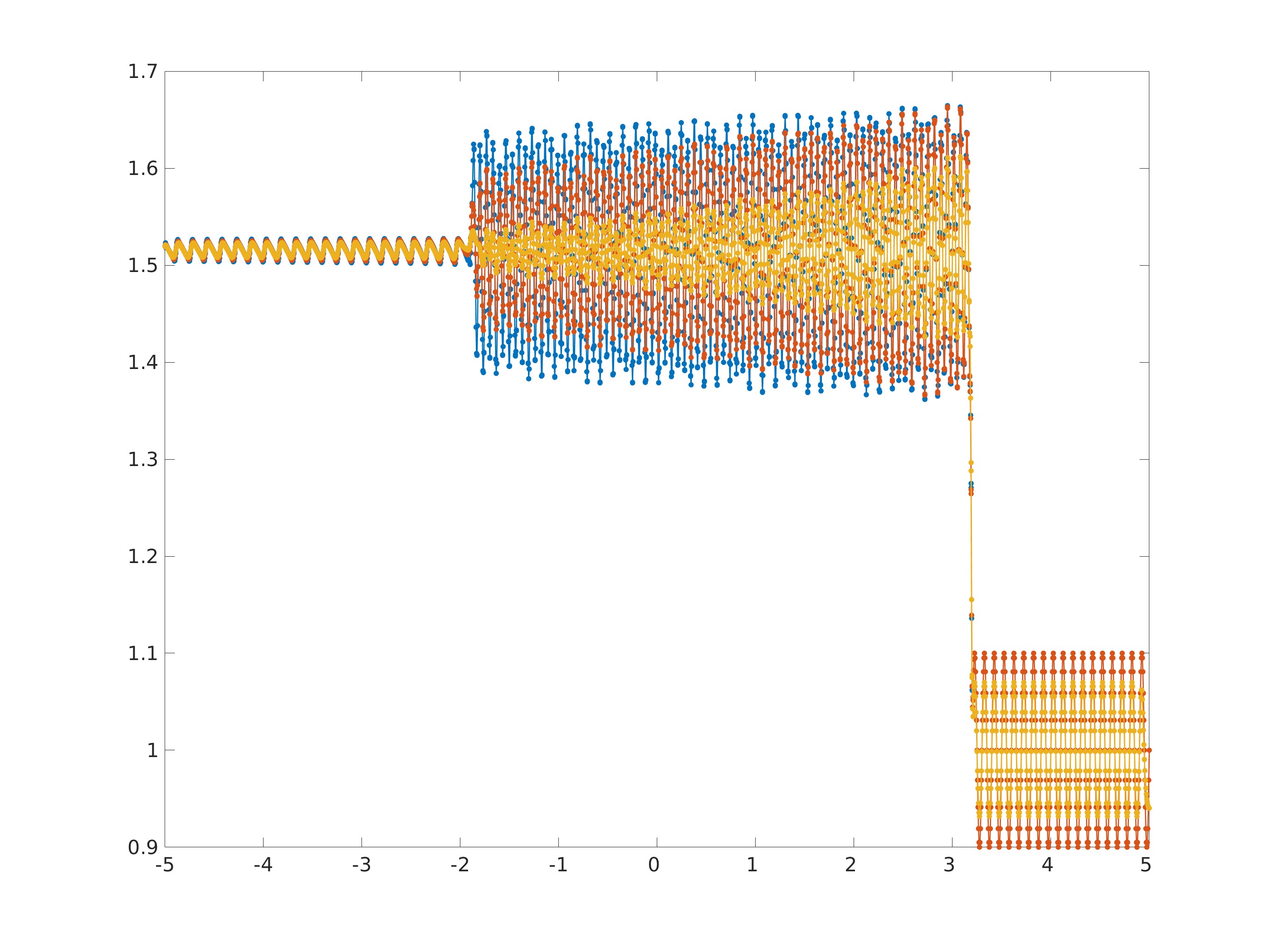}
		\caption{$p=2$}
		\label{fig:ttp2}
	\end{subfigure}
	\caption{Titarev-Toro problem, density profiles $\varrho$ at $t=5.0$ obtained using $E_h=1000$ and $p\in\{1,2\}$.}
	\label{fig:ttp}
\end{figure}

We perform numerical simulations up to the final time $t=5.0$ using $E_h=1000$ uniform elements. Fig. \ref{fig:ttp1} illustrates that both the DG-WENO and DG-RB-WENO schemes struggle to accurately capture the physical oscillations present in the exact solution when employing linear finite elements. In fact, the results obtained with the DG-WENO and the DG-RB-WENO-$1.0$ scheme are nearly indistinguishable. Clearly, increasing the threshold parameter $\theta$ effectively reduces numerical dissipation, thereby enhancing overall accuracy. Similar behavior can be observed for quadratic finite elements, as shown in Fig. \ref{fig:ttp2}. In contrast to the DG-WENO scheme, our residual-based variant accurately captures all features of the exact solution, even for relatively small values of $\theta$.
\subsubsection{Kelvin-Helmholtz instability}
Next, we consider a Kelvin-Helmholtz instability problem \cite{ma2023}. This scenario involves the transformation of a narrow shear layer into a complex pattern of vortices, serving as a benchmark to study shear-driven turbulent flows and assess a scheme's ability to resolve small-scale structures. It serves as an ideal test case to compare the dissipation characteristics of our DG-RB-WENO scheme with the DG-WENO scheme. The initial data
\begin{align*}
\begin{bmatrix}
\varrho_1 \\v_{x,1}\\v_{y,1}\\p_1
\end{bmatrix}=
\begin{bmatrix}\phantom{-}2.0\\-0.5\\0.01\sin(2\pi(x-0.5))\\\phantom{-}2.5
\end{bmatrix}, \quad 
\begin{bmatrix}
\varrho_2 \\v_{x,2}\\v_{y,2}\\p_2
\end{bmatrix}=
\begin{bmatrix}1.0\\0.5\\0.01\sin(2\pi(x-0.5))\\2.5
\end{bmatrix}
\end{align*}
are prescribed in the computational domain $\Omega=(0,1)\times(0,1)$.

In Fig. \ref{fig:kh}, we present snapshots of the density distribution at the final time $t=1.0$ obtained using $E_h=512^2$ elements and $p=1$. Clearly, the DG-WENO scheme suffers from excessive diffusion, thus preventing the resolution of fine-scale vortical structures. In contrast, the RB-DG-WENO scheme, by design, effectively mitigates this issue, allowing for the generation and interaction of fine-scale vortical structures.
\begin{figure}[!htb]
	\centering
	\begin{subfigure}[b]{.32\linewidth}
		\includegraphics[width=\linewidth]{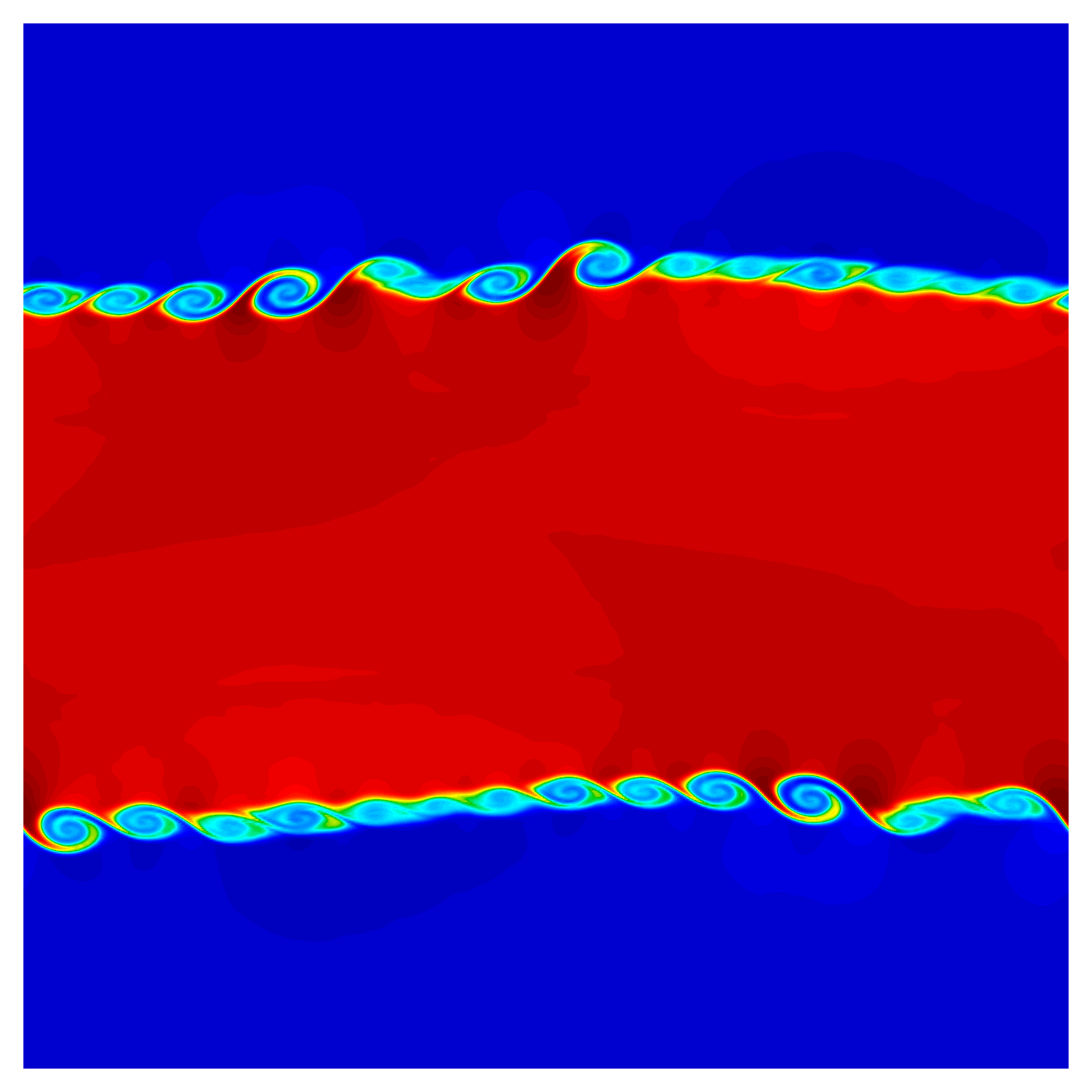}
		\caption{DG-WENO,\newline\hspace*{0.05cm} $u_h \in [0.964,2.116]$}
		\label{fig:kh1}
	\end{subfigure}
	\begin{subfigure}[b]{.32\linewidth}
		\includegraphics[width=\linewidth]{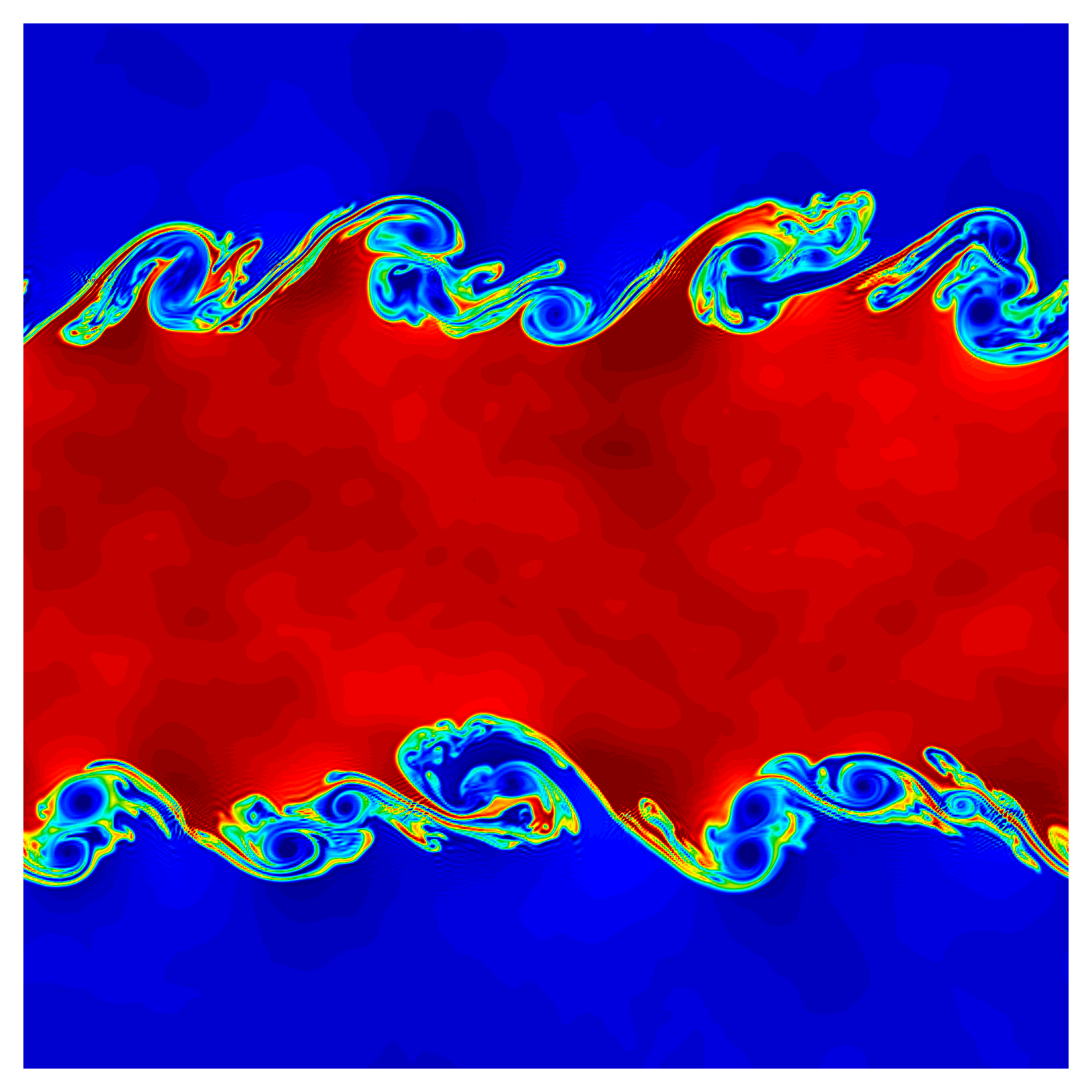}
		\caption{DG-RB-WENO-$0.1$,\newline \hspace*{-0.5cm}$u_h \in [0.183,3.005]$}
		\label{fig:kh2}
	\end{subfigure}
	\begin{subfigure}[b]{.32\linewidth}
		\includegraphics[width=\linewidth]{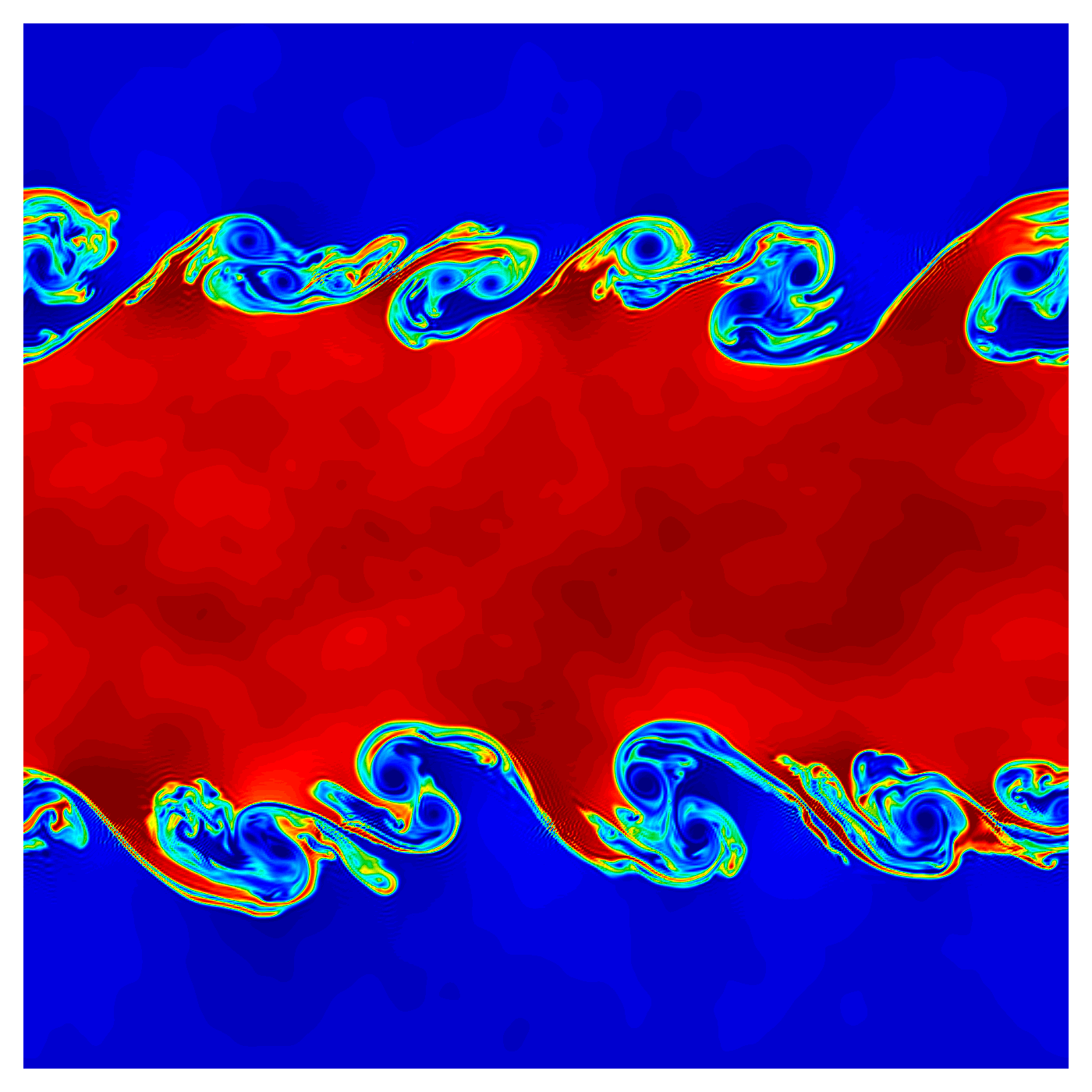}
		\caption{DG-RB-WENO-$1.0$,\newline\hspace*{-0.65cm} $u_h \in [0.232,3.304]$}
		\label{fig:kh3}
	\end{subfigure}
	\caption{Kelvin-Helmholtz instability, density profiles $\varrho$ at $t=1.0$ obtained using $E_h=512^2$ and $p=1$.}
	\label{fig:kh}
\end{figure}

\subsubsection{Double Mach reflection}
In our last numerical example, we investigate the double Mach reflection problem of Woodward and Colella \cite{woodward1984}. The flow pattern involves a propagating Mach $10$ shock in air which initially makes a $60^{\circ}$ angle with a reflecting wall. The computational domain is the rectangle $\Omega=(0,4)\times(0,1)$. The initial and boundary conditions are specified using the pre- and post-shock values as follows:
\begin{align*}
\begin{bmatrix}
\varrho_L \\v_{x,L}\\v_{y,L}\\p_L
\end{bmatrix}=
\begin{bmatrix}8.0\\\phantom{-}8.25\cos({30}^{\circ})\\-8.25\sin({30}^{\circ})\\116.5
\end{bmatrix}, \quad 
\begin{bmatrix}
\varrho_R \\v_{x,R}\\v_{y,R}\\p_R
\end{bmatrix}=
\begin{bmatrix}1.4\\0.0\\0.0\\1.0
\end{bmatrix}.
\end{align*}
Initially, the post-shock values (subscript L) are prescribed in the subdomain $\Omega_L=\{(x,y)\;|\;x<\frac{1}{6}+\frac{y}{\sqrt{3}}\}$ and the pre-shock values (subscript R) in $\Omega_R = \Omega \setminus \Omega_L$. The reflecting wall corresponds to $1/6 \leq x \leq 4$ and $y=0$. No boundary conditions are required along the line $x=4$. On the rest of the boundary, post-shock conditions are assigned for $x<\frac{1}{6}+\frac{1+20t}{\sqrt{3}}$ and pre-shock conditions elsewhere. The so-defined values along the top boundary describe the exact motion of the initial Mach $10$ shock.

The results shown in Figs \ref{fig:dmr1}-\ref{fig:dmr3} are obtained using $E_h=768\cdot192$ elements and $p=2$. Again, the DG-RB-WENO scheme demonstrates superior performance by capturing all features sharply and resolving the triple point region more accurately than the DG-WENO scheme. Moreover, it introduces sufficient numerical dissipation to suppress spurious oscillations near the shock. The generation of vortices within the triple point region is clearly visible for $\theta=1.0$, as shown in Fig. \ref{fig:dmr3}. 

\begin{figure}[!htb]
	\centering
	\begin{subfigure}[b]{.96\linewidth}
		\includegraphics[width=\linewidth]{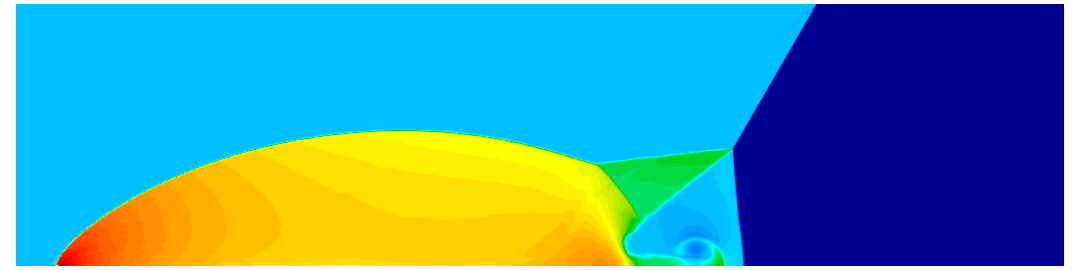}
		\caption{DG-WENO, $u_h \in [0.553,23.262]$}
		\label{fig:dmr1}
	\end{subfigure}
	\begin{subfigure}[b]{.96\linewidth}
		\includegraphics[width=\linewidth]{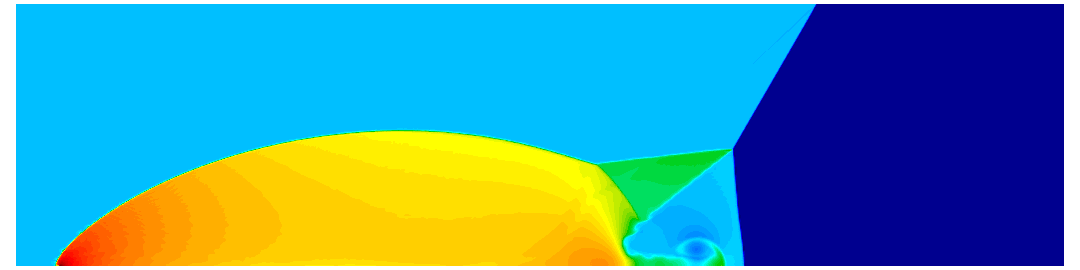}
		\caption{DG-RB-WENO-$0.1$, $u_h \in [0.757,23.700]$}
		\label{fig:dmr2}
	\end{subfigure}
	\begin{subfigure}[b]{.96\linewidth}
		\includegraphics[width=\linewidth]{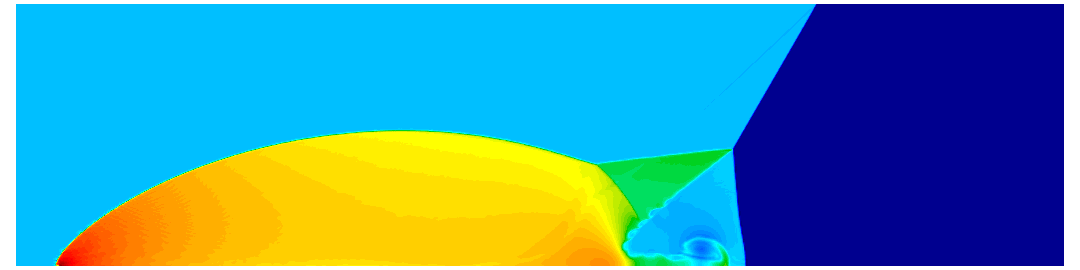}
		\caption{DG-RB-WENO-$1.0$, $u_h \in [0.526,24.501]$}
		\label{fig:dmr3}
	\end{subfigure}
	\caption{Double Mach reflection, density profiles $\varrho$ at $t=0.2$ obtained using $E_h=768\cdot192$ and $p=2$.}
	\label{fig:dmr}
\end{figure}
\section{Conclusions}
\label{sec:concl}
% ---------------------------------------------------------------------------
We presented a strongly consistent dissipation-based WENO stabilization scheme for finite element discretizations of hyperbolic conservation laws and systems thereof. The proposed methodology is suitable for CG and DG approximations alike. Following \cite{kuzmin2023a,vedral2023}, our approach utilizes piecewise-constant blending functions to combine low-order and high-order numerical dissipation operators, effectively achieving high-order accuracy while suppressing spurious oscillations both globally and locally. In our residual-based HWENO scheme, the weights for candidate polynomials in a WENO reconstruction are determined not only by the smoothness of the numerical approximation but also by the residual. We have shown how coercivity w.r.t.~a proper norm can be enforced for a particular kind of local projection stabilization. For scalar convection-diffusion-reaction equations, we proved that our scheme is stable, delivers at least one solution, and converges at the rate 1/2 in the worst case. Additionally, we presented an a posteriori criterion for achieving the optimal rate $k+1/2$. We envisage that this criterion can be used to construct local error indicators for adaptive mesh refinement purposes.

Our methodology allows for individual modification of various building blocks, such as stabilization operators, smoothness indicators, and reconstruction procedures, enabling further enhancements in accuracy for both CG and DG settings. To enhance accuracy on coarse meshes, we plan to develop smoothness sensors with a polynomial structure within each element, moving away from the piecewise constant approach. Finally, our scheme can be tailored to enforce inequality constraints, such as discrete maximum principles, by incorporating flux/slope limiters as discussed in \cite{dobrev2018,kuzmin2023,zhang2011}.

\section*{Acknowledgements}
D.\ Kuzmin and J.\ Vedral have been supported by the German Research Foundation (Deutsche
Forschungsgemeinschaft, DFG) under grant KU 1530/23-3.

A.\ Rupp has been supported by the Academy of Finland's grant number 354489 \emph{Uncertainty quantification for PDEs on hypergraphs}, grant number 359633 \emph{Localized orthogonal decomposition for high-order, hybrid finite elements}, Business Finland's project number 539/31/2023 \emph{3D-Cure: 3D printing for personalized medicine and customized drug delivery}, and the Finnish \emph{Flagship of advanced mathematics for sensing, imaging and modeling}, decision number 358944.
	
	\bibliographystyle{plain}
	\bibliography{paper_weno_analysis}
\end{document}